\documentclass{amsart}
\pdfoutput=1
\usepackage{amsmath,amssymb,amsthm,mathrsfs}
\usepackage{enumerate}
\usepackage[colorlinks=true,pagebackref,bookmarks,bookmarksdepth=4]{hyperref}
\usepackage{verbatim}
\usepackage[all]{xy}

\usepackage[T1]{fontenc}

\hypersetup{
 pdfinfo={
    Title={Additive relative invariants and the components of a linear free divisor},
   Author={Brian Pike}
 }
}

\newtheorem{theorem}{Theorem}[section]
\newtheorem*{theorem*}{Theorem}
\newtheorem{corollary}[theorem]{Corollary}
\newtheorem*{corollary*}{Corollary}
\newtheorem{proposition}[theorem]{Proposition}
\newtheorem*{proposition*}{Proposition}
\newtheorem{lemma}[theorem]{Lemma}
\newtheorem{conjecture}[theorem]{Conjecture}
\theoremstyle{definition}
\newtheorem{definition}[theorem]{Definition}
\theoremstyle{definition}
\newtheorem{example}[theorem]{Example}
\newtheorem{remark}[theorem]{Remark}

\numberwithin{equation}{section}

\newcommand{\C}{{\mathbb C}}
\newcommand{\Z}{{\mathbb Z}}
\newcommand{\N}{{\mathbb N}}
\newcommand{\A}{{\mathscr A}}
\newcommand{\g}{{\mathfrak{g}}}
\newcommand{\radu}{{\mathfrak{r}}}
\newcommand{\radg}{{\mathfrak{rad}\,\g}}
\newcommand{\lfrak}{{\mathfrak{l}}}
\newcommand{\bfrak}{{\mathfrak{b}}}
\newcommand{\tfrak}{{\mathfrak{t}}}
\newcommand{\z}{{\mathfrak{z}}}
\newcommand{\h}{{\mathfrak{h}}}
\newcommand{\fl}{{\mathfrak{l}}}
\newcommand{\gl}{{\mathfrak{gl}}}

\newcommand{\ad}{{\mathrm{ad}}}
\newcommand{\GL}{{\mathrm{GL}}}
\newcommand{\SL}{{\mathrm{SL}}}

\newcommand{\Hom}{{\mathrm{Hom}}}

\newcommand{\Ocnp}{{\mathscr{O}_{\C^n,p}}}

\newcommand{\Der}{{\mathrm{Der}}}
\newcommand{\Derlog}[2]{{\mathrm{Der}}_{#1}(-\log #2)}
\newcommand{\codim}{{\mathrm{codim}}}
\newcommand{\rank}{{\mathrm{rank}}}
\newcommand{\smooth}{{\mathrm{smooth}}}

\newcommand{\Rep}{\textrm{Rep}}
\newcommand{\Rad}{\mathrm{Rad}}
\newcommand{\gpcenter}{\mathrm{Z}}
\newcommand{\nsub}{\mathrel{\unlhd}}
\newcommand{\associatedmult}{\stackrel{\mathrm{m}}{\longleftrightarrow}}
\newcommand{\associatedadd}{\stackrel{\mathrm{a}}{\longleftrightarrow}}
\newcommand{\Gm}{{\mathbb{G}_\mathrm{m}}}
\newcommand{\Ga}{{\mathbb{G}_\mathrm{a}}}

\title[Additive relative invariants and linear free divisors]{Additive relative invariants and the components of a linear free divisor}
\date{\today}
\author{Brian Pike}
\address{Dept.\ of Computer and Math\-ematical Sciences,
University  of Tor\-onto Scarborough, 
1265 Military Trail, 
Toronto, ON M1C 1A4,
Canada}
\email{bapike@gmail.com}



\begin{document}
\begin{abstract}
A \emph{prehomogeneous vector space} is a
rational
representation
$\rho:G\to\GL(V)$ of a connected complex linear algebraic group $G$
that has a Zariski open orbit $\Omega\subset V$. 
Mikio Sato showed that 
the hypersurface components of $D:=V\setminus \Omega$
are related to the rational characters $H\to\GL(\C)$ of $H$,
an algebraic abelian quotient of $G$. 
Mimicking this work,
we investigate the \emph{additive functions} of
$H$, the homomorphisms $\Phi:H\to (\C,+)$.
Each such $\Phi$ is related to an
\emph{additive relative invariant},
a rational function $h$ on $V$ such that 
$h\circ \rho(g)-h=\Phi(g)$ on $\Omega$ for all $g\in G$.
Such an $h$ is homogeneous of degree $0$, and helps describe the behavior
of certain subsets of $D$ under the $G$--action.

For those prehomogeneous vector spaces with 
$D$ a type of hypersurface called a linear free divisor,
we prove there are no nontrivial additive functions of $H$, and hence
$H$ is an algebraic torus.
From this we gain insight into the structure of such representations and prove
that the number of irreducible components of $D$ equals the dimension
of the abelianization of $G$.
For some special cases
($G$ abelian, reductive, or solvable, or $D$ irreducible) 
we simplify proofs of existing results.
We also examine the homotopy groups of $V\setminus D$.
\end{abstract}
\maketitle
\tableofcontents

%
%
%
%
%
%
%
%
%
%
%
%
\section*{Introduction}
\label{sec:intro}
A \emph{prehomogeneous vector space} is a rational representation
$\rho:G\to\GL(V)$ of a connected complex linear algebraic group $G$
that has a (unique) Zariski open orbit $\Omega$ in $V$.
These representations have been much studied from the viewpoint of
number theory and harmonic analysis (e.g., \cite{kimura}).
A particularly useful tool is to consider
the rational
\emph{characters} $\chi:H\to \Gm:=\GL(\C)$
of the abelian quotient
$H:=G/[G,G]\cdot G_{v_0}$, where $G_{v_0}$ is the isotropy subgroup at
any $v_0\in\Omega$.
There is a correspondence between these characters and
`relative invariants';
a \emph{relative invariant} $h$ is a rational function on $V$ such
that on $\Omega$, $h$ is holomorphic and $h\circ\rho$ is merely $h$
multiplied by a character of $G$ (see \eqref{eqn:relinv}).
The relative invariants are constructed from the
irreducible polynomials defining the 
hypersurface components of the algebraic set $V\setminus \Omega$.
Thus, for example,
the number of hypersurface components of $V\setminus
\Omega$ equals the rank of the free abelian group of characters of $H$.

In this paper we study the \emph{additive functions} of $H$,
the rational homomorphisms $\Phi:H\to \Ga:=(\C,+)$.
This complements the classical study of the characters, as
up to isomorphism
$\Gm$ and $\Ga$ are the only
$1$-dimensional
connected complex linear algebraic groups.
Since 
$H$ is connected and abelian,
$H\cong\Gm^k\times \Ga^\ell$ with
$k$ the rank of the character group of $H$ and
$\ell$ the dimension of the vector space of 
additive functions of $H$;
the characters and additive functions thus describe $H$ completely.
Of particular interest for us,  
the number of irreducible hypersurface components of
$V\setminus \Omega$ equals
\begin{equation}
\label{eqn:kdimhell}
k=\dim(H)-\ell
\end{equation}
(see Corollary \ref{cor:numcomponents}),
where $\dim(H)$ is easily computable.

The content of the paper is as follows.
After reviewing the basic properties of prehomogeneous vector spaces
in \S\ref{sec:intropvs} 
and abelian linear algebraic groups in 
\S\ref{sec:abelian}, we study the additive functions of $H$ in
\S\ref{sec:addrelinv}.
Just as the 
characters of $H$ are related to relative invariants,
we show that
the additive functions of $H$ are related to `additive relative
invariants'
(see Definition \ref{defn:addrelinv}).
By Proposition \ref{prop:whichrationalfunctions},
a rational function $k$ on $V$
is an \emph{additive relative invariant} if and only if
$k(\rho(g)(v))-k(v)$ is independent of $v\in \Omega$ for all $g\in G$,
and
then this difference gives the additive function $G\to\Ga$ associated to
$k$.  
These additive relative invariants are homogeneous of degree
$0$ and of the form $\frac{h}{f}$ for a polynomial relative invariant $f$
(Proposition \ref{prop:homogeneous}).
Geometrically, an additive relative invariant $\frac{h}{f}$ describes a
$G$--invariant subset $f=h=0$ of $f=0$, with $G$ permuting
the sets 
$$V_{\epsilon}=\{x: f(x)=h(x)-\epsilon=0\},\qquad \epsilon\in\C,$$
by $\rho(g)(V_{\epsilon})=V_{\chi(g)\cdot \epsilon}$ for 
$\chi$ the character corresponding to the relative invariant
$f$
(see Proposition \ref{prop:funsagreeon});
the converse also holds if $\deg(\frac{h}{f})=0$ and $f$ is reduced.
In \S\ref{subsec:vanishing} we investigate which
characters and additive functions
vanish on which isotropy subgroups,
a key ingredient for \S\ref{sec:lfds}.
In \S\ref{subsec:algind} we establish the algebraic independence of a
generating set of relative invariants and the numerators of a basis of
the additive relative invariants.
Finally, in $\S\ref{subsec:decomp}$ we conjecture 
that
every additive relative invariant may be written as a sum of
additive relative invariants of 
the form $\frac{h_i}{f_i}$, where $f_i$ is a power of an irreducible
polynomial relative invariant.
We prove this splitting behavior
when the additive relative invariant may be written as a sum of
appropriate fractions.

In \S\ref{sec:exs} we describe some examples of prehomogeneous
vector spaces and their additive relative invariants.

In \S\ref{sec:lfds}, we study prehomogeneous vector spaces
with the property that
$D:=V\setminus \Omega$ is a type of hypersurface called a
\emph{linear free divisor}.
Such objects have been of much interest recently
(e.g., \cite{mondbuchweitz,gmns,freedivisorsinpvs,kpi1,DP-matrixsingI}),
and were our original motivation.
Using a criterion due to Brion and
the results of \S\ref{subsec:vanishing},
we show in Theorem \ref{thm:homomorphism}
that such prehomogeneous vector spaces have
no nontrivial additive relative invariants or nontrivial additive
functions.
Then
by \eqref{eqn:kdimhell}
the number of irreducible components of $D$
equals $\dim(H)$, and this may be easily computed 
using only the Lie algebra $\g$ of $G$
(Theorem \ref{thm:ofhomomorphism}; see also Remark
\ref{rem:worksmoregenerally}).
We also gain insight into the structure of $G$ and
the behavior of the $G$--action
(\S\S\ref{subsec:mainthm}, \ref{subsec:structureg},
\ref{subsec:liealgebra}).
In \S\ref{subsec:specialcases}, we study the special cases
where $G$ is abelian, reductive, or solvable, or where $D$ is
irreducible,
simplifying
proofs of several known results.
In \S\ref{subsec:whatelse}
we observe that, although the number of components of a linear free
divisor is computable from the abstract Lie algebra structure
$\g$ of $G$, the degrees of the polynomials defining these components are
not,
and seem to 
require the representation $\g\to\gl(V)$.
Lemma \ref{lemma:degrees} can compute these 
degrees, but requires data that are themselves difficult to
compute.
Finally, in \S\ref{subsec:homotopy} we use our earlier results to
study the homotopy groups of 
the complement of a linear free divisor.

\textbf{Acknowledgements:}
We would like to thank Mathias Schulze for very helpful
discussions that led to this work.
He deserves much of the credit for
Corollary \ref{cor:solvable}.
We also thank Christian Barz for finding an error in an example.

\section{Prehomogeneous vector spaces}
\label{sec:intropvs}
We begin by briefly reviewing 
prehomogeneous vector spaces.
Our reference is 
\cite{kimura}, although the results we describe
were developed by 
M.~Sato in the 1960's (\cite{sato}).

In the whole article, we shall study only complex linear algebraic
groups.
For such a $K$,
let $K^0$ denote the connected component of $K$
containing the identity,
$L(K)$ the Lie algebra of $K$,
and $K_v$ the isotropy subgroup at $v$ of a particular $K$--action. 

Let $V$ be a finite-dimensional complex vector space, and $G$ a
connected complex linear algebraic group.  Let $\rho:G\to\GL(V)$ be a
rational representation of $G$, i.e., a homomorphism of linear algebraic
groups.
When $G$ has an open orbit $\Omega$ in $V$, then we call $(G,\rho,V)$
a \emph{prehomogeneous vector space}.  Then $\Omega$ is
unique and Zariski open, so that the complement $\Omega^c=V\setminus \Omega$ is an
algebraic set in $V$.  We call $\Omega^c$ the \emph{exceptional orbit
variety} as it is the union of the non-open orbits of $G$; 
others use \emph{discriminant} or \emph{singular
set}.

One of the basic theorems of prehomogeneous vector spaces is that the
hypersurface components of $\Omega^c$ may be detected from certain
multiplicative characters of $G$.
More precisely,
for any complex linear
algebraic group $K$ let $X(K)$ denote the set of rational
(multiplicative) \emph{characters}, that is, the
homomorphisms $K\to \Gm:=\GL(\C)$ of linear algebraic groups.
Then a rational function $f$ on $V$ that is not identically $0$ is a
\emph{relative invariant}
if there exists a $\chi\in X(G)$ such that
\begin{equation}
\label{eqn:relinv}
f(\rho(g)(v))=\chi(g)\cdot f(v)
\end{equation}
for all $v\in \Omega$ and $g\in G$, in which case
we\footnote{Usually this is written 
$f\longleftrightarrow \chi$, but we shall use similar notation for
the relationship between
additive functions and additive relative invariants.}
write
$f\associatedmult \chi$.
By \eqref{eqn:relinv}, the zeros and poles of
$f$ may occur only on $\Omega^c$.

Actually, $f$ and $\chi$ provide almost the same information
when $f\associatedmult \chi$.
Using $f$, we may
choose any $v_0\in \Omega$ and then recover $\chi$ by defining 
$$\chi(g)=\frac{f(\rho(g)(v_0))}{f(v_0)}\in(\C^*,\cdot)\cong\Gm,\qquad g\in G.$$
Conversely, using $\chi$ we may recover a nonzero constant multiple
$h$ of $f$ by choosing any $v_0\in\Omega$ and any nonzero value
$h(v_0)\in\C$,
defining $h$ on $\Omega$ by $h(\rho(g)(v_0))=\chi(g) h(v_0)$,
and then using the density of $\Omega$ to extend $h$ uniquely to 
a rational function on $V$.

Let
$X_1(G)$ be the set of $\chi\in X(G)$ for which there exists an $f$
with $f\associatedmult\chi$.
Then $X_1(G)$ is an abelian group: if
$f_i\associatedmult \chi_i$ and $n_i\in \Z$ for $i=1,2$, then
$(f_1)^{n_1}(f_2)^{n_2}\associatedmult (\chi_1)^{n_1}(\chi_2)^{n_2}\in
X_1(G)$.
The group $X_1(G)$ contains information 
about the hypersurface
components of $\Omega^c$.
\begin{theorem}[e.g., {\cite[Theorem 2.9]{kimura}}] 
\label{thm:pvs1}
Let $(G,\rho,V)$ be a prehomogeneous vector space with exceptional
orbit variety $\Omega^c$, and $S_i=\{x\in V: f_i(x)=0\}$,
$i=1,\ldots,r$ the distinct irreducible hypersurface components of
$\Omega^c$.  Then the irreducible polynomials $f_1,\ldots,f_r$ are
relative invariants which are algebraically independent, and any
relative invariant is of the form $c\cdot (f_1)^{m_1}\cdots (f_r)^{m_r}$ for
nonzero
$c\in\C$, and $m_i\in\Z$.  Moreover, if each $f_i\associatedmult \chi_i$,
then $X_1(G)$ is a free abelian group of rank $r$ generated by
$\chi_1,\ldots,\chi_r$.
\end{theorem}

The (homogeneous) irreducible polynomials $f_1,\ldots,f_r$ of the Theorem are called the \emph{basic
relative invariants} of $(G,\rho,V)$.
\begin{remark}
\label{rem:omegac}
Since $v\in\Omega$ if and only if the corresponding orbit map $g\mapsto
\rho(g)(v)$ is a submersion at $e\in G$,
and this derivative depends linearly on $v\in V$,
the set $\Omega^c$ may be defined by an ideal generated by
homogeneous polynomials of degree $\dim(V)$
(see \S\ref{subsec:lfdsintro}). 
\end{remark}

There exists another description of $X_1(G)$.  Fix an element $v_0\in
\Omega$ with isotropy subgroup $G_{v_0}$.
Since the orbit of
$v_0$ is open, we have $\dim(G_{v_0})=\dim(G)-\dim(V)$.
Define the algebraic groups
$$
G_1=[G,G]\cdot G_{v_0}\qquad\textrm{and}\qquad
H=G/G_1.$$
The group $G_1$ does not depend on the choice of $v_0\in\Omega$ as all
such
isotropy subgroups are conjugate,
and $[G,G]\subseteq G_1$.
By \eqref{eqn:relinv}, every $\chi\in X_1(G)$ has
$G_1\subseteq \ker(\chi)$, and thus factors through the quotient to
give a corresponding
$\widetilde{\chi}\in
X(H)$.
In fact, the map $\chi\mapsto \widetilde{\chi}$ is an isomorphism:
\begin{proposition}[e.g., {\cite[Proposition 2.12]{kimura}}]
\label{prop:x1is}
Let $(G,\rho,V)$ be a prehomogeneous vector space with open orbit
$\Omega$.  Then for any $v_0\in \Omega$,
$$X_1(G)\cong X(G/[G,G]\cdot G_{v_0}).$$
\end{proposition}
Consequently, the rank of the character group $X(H)$
equals 
the number of irreducible hypersurface components of
$V\setminus \Omega$.

\section{Abelian complex linear algebraic groups}
\label{sec:abelian}
Let $(G,\rho,V)$ be a prehomogeneous vector space, and use the
notation of the previous section
with $H=G/[G,G]\cdot G_{v_0}$ for some $v_0\in\Omega$.
To understand the
rank of $X(H)$ and thus the 
number of
irreducible hypersurface components of the exceptional orbit variety
$\Omega^c$, we must understand $H$.  Since $H$ is
abelian and connected, its structure is very simple.
Recall that over $\C$ there are exactly two distinct
$1$-dimensional connected linear algebraic
groups, $\Gm=\GL(\C^1)\cong(\C^*,\cdot)$ and 
$\Ga=(\C,+)$.

\begin{proposition}[e.g., {\cite[\S3.2.5]{ov}}]
\label{prop:ckcl}
An abelian connected complex linear algebraic group $K$ is isomorphic to
$(\Gm)^k\times (\Ga)^\ell$ for nonnegative integers $k$ and $\ell$,
where the exponents denote a repeated direct product.
\end{proposition}

Note that this Proposition is false for some other fields.

For any linear algebraic group $K$, let $\A(K)$ be the
complex vector space of rational
homomorphisms $\Phi:K\to \Ga$, sometimes called \emph{additive
functions of $K$} (\cite[\S3.3]{springer}).  
When $K$ is connected, $\Phi\in\A(K)$ is determined by $d\Phi_{(e)}$,
and hence
$\A(K)$ is
finite-dimensional. 
When $K$ has the decomposition of Proposition \ref{prop:ckcl},
the rank of $X(K)$ and the dimension of $\A(K)$ are related
to $k$ and $\ell$.

\begin{lemma}
\label{lemma:rankdim}
If $K=(\Gm)^k\times(\Ga)^\ell$, then $X(K)$ is free of rank $k$ and
$\dim_\C(\A(K))=\ell$.
\end{lemma}
\begin{proof}
By the Jordan decomposition, any $\chi\in X(K)$ will have
$\{1\}^k\times (\Ga)^\ell\subseteq \ker(\chi)$, so $\chi$ factors through to an
element of $X((\Gm)^k)$.
Thus $X(K)\cong X((\Gm)^k)$, and similarly $\A(K)\cong
\A((\Ga)^\ell)$.
Finally, an easy exercise shows that $\rank(X((\Gm)^k))=k$, and
$\dim(\A((\Ga)^\ell))=\ell$.
\end{proof}

Thus, for a prehomogeneous vector space we have:

\begin{corollary}
\label{cor:numcomponents}
Let $(G,\rho,V)$ be a prehomogeneous vector space with open orbit
$\Omega$.  Let $v_0\in\Omega$ and $H=G/[G,G]\cdot G_{v_0}$.
Then the number of
irreducible hypersurface components of $\Omega^c$ equals
$\dim(H)-\dim(\A(H))$.
\end{corollary}
\begin{proof}
By Theorem \ref{thm:pvs1}, Proposition \ref{prop:x1is},
Proposition \ref{prop:ckcl}, and Lemma \ref{lemma:rankdim},
the number of irreducible components equals
$\rank(X_1(G))=\rank(X(H))=\dim(H)-\dim(\A(H))$.
\end{proof}

It is natural, then, to study
these additive functions and their geometric meaning.

\section{Additive relative invariants}
\label{sec:addrelinv}

In this section, we develop for additive functions the analogue of
the multiplicative theory from
\S \ref{sec:intropvs}.
We encourage the reader to consult the examples of \S\ref{sec:exs} as
needed. 
Let $(G,\rho,V)$ be a
prehomogeneous vector space with open orbit $\Omega$.
Let $v_0\in\Omega$
and $H=G/[G,G]\cdot G_{v_0}$.

  \subsection{Definition}
We define additive relative invariants 
similarly to (multiplicative) relative invariants.
\begin{definition}
\label{defn:addrelinv}
A rational function $h$ on $V$ is an
\emph{additive relative invariant} of $(G,\rho,V)$ if there exists a $\Phi\in \A(G)$
so that
\begin{equation}
\label{eqn:star}
h(\rho(g)(v))-h(v)=\Phi(g)
\end{equation}
for all $v\in\Omega$ and $g\in G$.  In this situation, we write 
$h\associatedadd \Phi$.
\end{definition}

By \eqref{eqn:star}, the poles of such an $h$ may occur
only on $\Omega^c$, and $h$ is constant on the orbits of $\ker(\Phi)$.

\subsection{Basic properties}

We now establish some basic facts about additive relative invariants.
First we investigate the uniqueness of the relationship $h\associatedadd
\Phi$.

\begin{proposition}
\label{prop:basicproperties}
Let $(G,\rho,V)$ be a prehomogeneous vector space.
\begin{enumerate}
\item
\label{enit:hdifferbyalpha}
If $h_1\associatedadd \Phi$ and $h_2\associatedadd \Phi$, then
there exists an $\alpha\in \C$ with $h_1=\alpha+h_2$.
\item
\label{enit:Phiareequal}
If $h\associatedadd \Phi_1$ and $h\associatedadd \Phi_2$, then
$\Phi_1=\Phi_2$.
\end{enumerate}
\end{proposition}
\begin{proof}
For (1), fix $v_0\in\Omega$ and let $\alpha=h_1(v_0)-h_2(v_0)$,  
so $\Phi(g)+h_1(v_0)=\alpha+\Phi(g)+h_2(v_0)$ for all $g\in G$.
Applying
\eqref{eqn:star} shows that $h_1=\alpha+h_2$ on $\Omega$, and thus on
$V$.

(2) is immediate from \eqref{eqn:star}.
\end{proof}

Let $\A_1(G)$ be the set of $\Phi\in \A(G)$ for which there exists
a rational function $h$ on $V$
with $h\associatedadd \Phi$.
We now identify the additive functions in $\A_1(G)$, analogous to 
Proposition \ref{prop:x1is}.
\begin{proposition}
\label{prop:a1}
As vector spaces, $\A_1(G)\cong \A(H)$, where
$H=G/[G,G]\cdot G_{v_0}$ and $v_0\in\Omega$.
\end{proposition}
\begin{proof}
Let $\Phi\in\A_1(G)$ with $h\associatedadd \Phi$.
Evaluating \eqref{eqn:star} at $v_0\in\Omega$ and $g\in G_{v_0}$ shows
that $G_{v_0}\subseteq \ker(\Phi)$.  Since $\Ga$ is abelian,
$[G,G]\subseteq \ker(\Phi)$.  Thus any $\Phi\in\A_1(G)$ factors
through
the quotient $\pi:G\to H$
to a unique $\overline{\Phi}\in\A(H)$,
with $\Phi=\overline{\Phi}\circ \pi$.
The map $\rho:\A_1(G)\to \A(H)$ defined by $\rho(\Phi)=\overline{\Phi}$
is $\C$-linear.

Conversely, define the $\C$-linear map $\sigma:\A(H)\to \A(G)$
by $\sigma(\overline{\Phi})=\overline{\Phi}\circ \pi$.
For $\overline{\Phi}\in \A(H)$,
we have $G_{v_0}\subseteq \ker(\overline{\Phi}\circ \pi)$, and hence
we may define a function
$\overline{h}:\Omega\to\C$ by
$\overline{h}(\rho(g)(v_0))=(\overline{\Phi}\circ \pi)(g)$.
By the argument%
\footnote{This uses the fact that $\C$ has characteristic $0$.}
of \cite[Proposition 2.11]{kimura},
$\overline{h}$ is 
a regular function on $\Omega$
that may be extended to 
a rational function $h$ on $V$.
By construction, $h(\rho(g)(v))-h(v)=(\overline{\Phi}\circ \pi)(g)$
for $v=v_0$ and all $g\in G$, but this implies that this equation holds for all $v\in\Omega$
and all $g\in G$.
Thus $h\associatedadd \overline{\Phi}\circ \pi$,
and $\sigma:\A(H)\to \A_1(G)$.

Finally, check that $\rho$ and $\sigma$ are inverses
using the uniqueness of the factorization through $\pi$.
\end{proof}

Additive relative invariants are the rational functions on $V$
having the following property. 

\begin{proposition}
\label{prop:whichrationalfunctions}
Let $h:V\to\C$ be a rational function on $V$, holomorphic on $\Omega$.
Then $h$ is an additive relative invariant if, and only if,
for all $g\in G$,
$$h(\rho(g)(v))-h(v)\textrm{ is independent of $v\in\Omega$.}$$
\end{proposition}
\begin{proof}
By \eqref{eqn:star}, additive relative invariants have this property.

Conversely, define
$\Phi:G\to\C$ by 
$\Phi(g)=h(\rho(g)(v))-h(v)$ for any $v\in \Omega$.
As a composition of regular functions, $\Phi$ is regular on $G$.
Let $g_1,g_2\in G$, and $v\in \Omega$.  Then
\begin{eqnarray*}
\Phi(g_1 g_2^{-1})
&=&h\!\left(\rho(g_1 g_2^{-1})(v)\right)-h(v) \\
&=&h\!\left(\rho(g_1)(\rho(g_2^{-1})(v))\right)-h\!\left(\rho(g_2^{-1})(v)\right)
 +h\!\left(\rho(g_2^{-1})(v)\right)-h(v) \\
&=&\Big(h(\rho(g_1)(v_1))-h(v_1)\Big)
  +\Big(h(v_1)-h(\rho(g_2)(v_1))\Big),
\end{eqnarray*}
where $v_1=\rho(g_2^{-1})(v)\in \Omega$.
By our hypothesis, 
$\Phi(g_1 g_2^{-1})=\Phi(g_1)-\Phi(g_2)$.
Thus $\Phi:G\to \Ga$ is a homomorphism of algebraic groups, and
since $h\associatedadd \Phi$
we have
$\Phi\in \A_1(G)$.
%
\end{proof}

  \subsection{Homogeneity}

Relative invariants are always homogeneous rational functions on $V$
because $\rho$ is a linear representation.
Similarly, additive relative invariants are
homogeneous of degree $0$.

\begin{proposition}
\label{prop:homogeneous}
If $h\associatedadd \Phi$ and $h$ is not the zero function,
then $h$ may be written as
$h=\frac{h_1}{f_1}$, where $h_1$ and $f_1$ are homogeneous polynomials
with $\deg(h_1)=\deg(f_1)$, $h_1$ and $f_1$ have no common factors,
and $f_1$ is a relative invariant.
\end{proposition}
\begin{proof}
First note that by Remark \ref{rem:omegac}, for any
$t\in\C^*$, we have
$v\in\Omega$ if and
only if $t\cdot v\in\Omega$; we use this fact implicitly in the
rest of the proof since Definition \ref{defn:addrelinv} describes the
behavior of $h$ only on $\Omega$.

Let $t\in\C^*$, and define the rational function $h_t$ on $V$ by
$h_t(v)=h(t\cdot v)$.  Since $\rho$ is a linear representation,
\begin{equation}
\label{eqn:homog1}
h_t(\rho(g)(v))=h(t\cdot \rho(g)(v))=h(\rho(g)(t\cdot v)).
\end{equation}
Applying \eqref{eqn:star} to \eqref{eqn:homog1} gives
\begin{equation}
\label{eqn:homog2}
h_t(\rho(g)(v))=\Phi(g)+h(t\cdot v)=\Phi(g)+h_t(v).
\end{equation}
Since \eqref{eqn:homog2} holds for all $g\in G$ and $v\in \Omega$,
we have $h_t\associatedadd \Phi$.

By Proposition \ref{prop:basicproperties}\eqref{enit:hdifferbyalpha},
there exists a function
$\alpha:\C^*\to \C$ such that
\begin{equation}
\label{eqn:homog3}
h(t\cdot v)=h(v)+\alpha(t)
\end{equation}
for all $t\in\C^*$ and $v\in \Omega$.
If $s,t\in\C^*$ and $v\in\Omega$, then using \eqref{eqn:homog3}
repeatedly shows 
\begin{equation*}
h(v)+\alpha(st)=h(s(t\cdot v))=h(t\cdot
v)+\alpha(s)=h(v)+\alpha(t)+\alpha(s),
\end{equation*}
or
$\alpha(st)=\alpha(s)+\alpha(t)$.
By \eqref{eqn:homog3}, we have $\alpha(1)=0$,
and hence $\alpha:(\C^*,\cdot)\to(\C,+)$ is a group homomorphism.

Fixing some $v\in\Omega$ and instead using \eqref{eqn:homog3} to define
$\alpha$ shows that
$\alpha:\Gm\to\Ga$ is regular, hence
a homomorphism of complex linear
algebraic groups.  By the Jordan decomposition,
$\alpha=0$.
Then $h(t\cdot v)=h(v)$ for all nonzero $t\in \C$ and $v\in\Omega$;
by density, $h$ is homogeneous
of degree $0$. 

Thus we may write $h=\frac{h_1}{f_1}$, with $h_1$ and $f_1$ homogeneous polynomials
of equal degree, and without common factors.  Since $h$ may only have
poles on $\Omega^c$, $f_1$ defines a hypersurface
in $\Omega^c$ or is a nonzero constant.
By Theorem \ref{thm:pvs1}, $f_1$ is a relative
invariant.
\end{proof}

This gives an apparently nontrivial result about the structure of
prehomogeneous vector spaces for which the exceptional orbit variety
has no hypersurface components.

\begin{corollary}
\label{cor:nontrivial}
Let $(G,\rho,V)$ be a prehomogeneous vector space, let
$v_0\in \Omega$,
and let $r$ be the number of irreducible hypersurface components of
$V\setminus \Omega$.
\begin{enumerate}
\item
\label{enit:X1is1}
If $r=0$,
then
$\A_1(G)=\{0\}$ and $G=[G,G]\cdot G_{v_0}$.
\item
\label{enit:ddimleq1}
If $d:=\dim(G/[G,G]\cdot G_{v_0})\leq 1$, then
$\A_1(G)=\{0\}$ and $r=d$.
\end{enumerate}
\end{corollary}

In particular, 
$V\setminus \Omega$ has no hypersurface components if and only if
$G=[G,G]\cdot G_{v_0}$ for $v_0\in\Omega$.

\begin{proof}
Suppose $r=0$, so $X_1(G)=\{1\}$.
Let $\Phi\in \A_1(G)$, and choose $h$ with  
$h\associatedadd \Phi$.
Assume that $\Phi$ is nonzero and hence $h$ is not constant.
By Proposition \ref{prop:homogeneous},
write $h=\frac{h_1}{f_1}$ for polynomials $f_1$ and $h_1$ with
$\deg(h_1)=\deg(f_1)>0$, and $f_1$ a relative invariant.
If $f_1\associatedmult \chi_1$, then by our hypothesis $\chi_1=1$.
By \eqref{eqn:relinv}, $f_1$ is a nonzero constant on $\Omega$,
contradicting the fact that $\deg(f_1)>0$.
Thus $\Phi=0$ and so $\A_1(G)=\{0\}$.
Since there are no nontrivial characters or additive functions,
by Proposition \ref{prop:ckcl} and Lemma \ref{lemma:rankdim}
we find that
$H=G/[G,G]\cdot G_{v_0}$ consists of a single point, hence
$G=[G,G]\cdot G_{v_0}$, proving \eqref{enit:X1is1}.

Now suppose $d\leq 1$.  If $d=0$, then 
by Corollary \ref{cor:numcomponents} and Proposition \ref{prop:a1},
$r=0-\dim(\A_1(G))\geq 0$, implying the statement.
If $d=1$, then by the same reasoning either $r=0$ or $r=1$, but the
former is impossible by \eqref{enit:X1is1}.
Now apply Corollary \ref{cor:numcomponents} to show $\A_1(G)=\{0\}$.
\end{proof}

  \subsection{Global equation}
Our definition of an additive relative invariant
only involves the behavior of the function on
$\Omega$.
We may describe the 
behavior on all of $V$.
\begin{proposition}
\label{prop:globaleqn}
Let $h\associatedadd \Phi$ with $h$ nonzero.
As in Proposition \ref{prop:homogeneous}, write
$h=\frac{h_1}{f_1}$ for polynomials $h_1$ and $f_1$,
with $f_1\associatedmult \chi$.
Then for all $g\in G$ and $v\in V$, 
\begin{equation}
\label{eqn:allgv}
h_1(\rho(g)(v))=\chi(g)\cdot 
 \left( f_1(v)\cdot \Phi(g)+h_1(v) \right).
\end{equation}
\end{proposition}
\begin{proof}
Define the regular functions
$L,R:G\times V\to \C$ using the left, respectively, right sides
of \eqref{eqn:allgv}.
By \eqref{eqn:star} and \eqref{eqn:relinv}, these functions agree on $G\times \Omega$, and
thus on $G\times V$.
\end{proof}

Of course, \eqref{eqn:allgv} also holds for
$h=0
\associatedadd \Phi=0$ if $h_1=0$ and $f_1$ is a nonzero constant.

\subsection{Geometric interpretation}
We may now provide a 
geometric interpretation of additive
relative invariants.
For polynomials $h_1,h_2,f$ on $V$ with $f$ irreducible and nonzero,
we will say that \emph{$h_1$ and $h_2$ agree to order $k$ on $V(f)$}
if $h_1-h_2\in (f)^{k+1}$, or equivalently, if in coordinates
$\frac{\partial^I}{\partial x^I}(h_1-h_2)(v)=0$ for all $v\in V(f)$
and all multi-indices $I$ with $|I|\leq k$.

\begin{proposition}
\label{prop:funsagreeon}
Let $f_1,\ldots,f_r$ be the basic relative invariants, with
$f_i\associatedmult \chi_i$.
Let $h=\frac{h_1}{f}$ for homogeneous polynomials $h_1$ and $f$ with
$\deg(h_1)=\deg(f)$, 
$h$ written in lowest terms, 
and $f=f_1^{k_1}\cdots f_r^{k_r}$ with $k_i\geq 0$. Let
$\chi=\chi_1^{k_1}\cdots \chi_r^{k_r}$.
Then the following are equivalent:
\begin{enumerate}
\item
\label{enit:haddrelinv}
$h$ is an additive relative invariant.
\item
\label{enit:hagreeorder}
For all $i=1,\ldots,r$ with $k_i>0$,
and all $g\in G$,
$h_1\circ \rho(g)$ and $\chi(g)\cdot h_1$ agree to order $k_i-1$ on
$V(f_i)$. 
\item
\label{enit:hdivisible}
For all $g\in G$, $h_1\circ \rho(g)-\chi(g)\cdot h_1\in\C[V]$ is divisible by
$f$.
\end{enumerate}
In particular, when $f$ is reduced and we let
$V_\epsilon=\{x: f(x)=h_1(x)-\epsilon=0\}$
for all $\epsilon\in \C$,
then
$h$ is an additive relative invariant if and only if
$\rho(g)(V_\epsilon)=V_{\chi(g)\cdot \epsilon}$ for all $g\in G$ and
$\epsilon\in\C$.
\end{proposition}
\begin{proof}
If \eqref{enit:haddrelinv}, then by \eqref{eqn:allgv},
for all $g\in G$ and all $i=1,\ldots,r$ we have
$$(h_1\circ \rho(g))-\chi(g)\cdot h_1\in (f_i)^{k_i},$$
implying \eqref{enit:hagreeorder}.

\eqref{enit:hagreeorder} implies \eqref{enit:hdivisible} because
$f_1,\ldots,f_r$ all define
distinct irreducible hypersurfaces. 

If \eqref{enit:hdivisible}, then for all $g,v$ we have 
\begin{equation}
\label{eqn:h1rhogvminus}
h_1(\rho(g)(v))-\chi(g)\cdot h_1(v)=f(v)\cdot \beta(g,v)
\end{equation}
for some $\beta\in \C[G\times V]$.
For each $g\in G$, both sides of 
\eqref{eqn:h1rhogvminus} should be zero or homogeneous of degree
$\deg(h_1)=\deg(f)$ as functions on $V$; hence,
$\beta(g,-)\in \C[V]$ is zero or
homogeneous of degree $0$, i.e., constant, and so $\beta\in \C[G]$.
Finally, rearrange \eqref{eqn:h1rhogvminus} and apply
\eqref{eqn:relinv} and Proposition \ref{prop:whichrationalfunctions}.

For the last part of the claim,
if $h$ is an additive relative invariant then
$\rho(g)(V_\epsilon)=V_{\chi(g)\cdot \epsilon}$ by \eqref{eqn:allgv}.
Conversely,
$\rho(g)(V_\epsilon)=V_{\chi(g)\cdot \epsilon}$
for all $\epsilon\in\C$
implies that
$h_1\circ \rho(g)-\chi(g)\cdot h_1$
vanishes on $V(f)$, and hence is divisible by $f$, i.e.,
\eqref{enit:hdivisible}.
\end{proof}

  \subsection{Vanishing lemma}
\label{subsec:vanishing}

We now prove that certain 
characters and additive functions
vanish on certain isotropy subgroups.
This is a key observation used in \S\ref{sec:lfds}.

\begin{lemma}
\label{lemma:newvanishingthm}
Let $v\in V$.
\begin{enumerate}[(i)]
\item
\label{enit:gvkerchi}
Let $f_1$ be a relative invariant, $f_1\associatedmult \chi$.
If $f_1$ is defined at $v$ and $f_1(v)\neq 0$, then $G_{v}\subseteq \ker(\chi)$.
\item
\label{enit:gvkerchiorphi}
Let $\frac{h_1}{f_1}$ be an additive relative invariant
written as in Proposition \ref{prop:homogeneous},
with $h_1$ nonzero, $\frac{h_1}{f_1}\associatedadd \Phi$, and
$f_1\associatedmult \chi$.
Then the following conditions are equivalent:
\begin{itemize}
\item $h_1(v)=0$ or $G_v\subseteq \ker(\chi)$,
\item $f_1(v)=0$ or $G_v\subseteq \ker(\Phi)$.
\end{itemize}
\end{enumerate}
\end{lemma}
\begin{proof}
Let $g\in G_v$.
For \eqref{enit:gvkerchi}, by 
\eqref{eqn:relinv} we have
$$
f_1(v)=f_1(\rho(g)(v))=\chi(g)f_1(v),
$$
and \eqref{enit:gvkerchi} follows.
For \eqref{enit:gvkerchiorphi}, by
Proposition \ref{prop:globaleqn} we have
$$
h_1(v)=h_1(\rho(g)(v))=\chi(g) \left(
  f_1(v) \Phi(g)+h_1(v) \right),
$$
or 
\begin{equation}
\label{eqn:h1vthingy}
h_1(v)\cdot \left(1-\chi(g)\right)
 = \chi(g)\cdot f_1(v) \cdot \Phi(g).
\end{equation}
%
Now consider the circumstances
when one side of \eqref{eqn:h1vthingy}
vanishes for all $g\in G_v$; the other side vanishes,
too.
\end{proof}

Lemma \ref{lemma:newvanishingthm} is best understood when
we consider generic points on the hypersurface components of
$\Omega^c$.
 
\begin{lemma}
\label{lemma:vanishingthm}
Let $f_1,\ldots,f_r$ be the basic relative invariants of $(G,\rho,V)$,
with $f_i\associatedmult \chi_i$.
Let $v_i\in V(f_i)$, with $v_i\notin V(f_j)$ for $i\neq j$.
\begin{enumerate}
\item
\label{enit:figv}
If $i\neq j$, then $G_{v_i}\subseteq \ker(\chi_j)$.
\item
\label{enit:vanishing2}
Suppose $h\associatedadd \Phi$, with $h$ nonzero and $h=\frac{h_1}{g_1}$ as in
Proposition \ref{prop:homogeneous}, and
$g_1=f_1^{k_1}\cdots f_r^{k_r}$ the factorization of $g_1$.
Then:
\begin{enumerate}
\item
\label{enit:gvichi}
If $k_i>0$ and $h_1(v_i)\neq 0$, then
$G^0_{v_i}\subseteq \ker(\chi_i)$.
\item
\label{enit:gviphi}
If $k_i=0$, then
$G_{v_i}\subseteq \ker(\Phi)$.
\end{enumerate}
\end{enumerate}
\end{lemma}
\begin{proof}
\eqref{enit:figv} follows from Lemma
\ref{lemma:newvanishingthm}\eqref{enit:gvkerchi}.
For \eqref{enit:gvichi},
apply Lemma \ref{lemma:newvanishingthm}\eqref{enit:gvkerchiorphi}
to show that $G_{v_i}\subseteq \ker(\chi_1^{k_1}\cdots \chi_r^{k_r})$.
Then by \eqref{enit:figv}, $G_{v_i}\subseteq \ker(\chi_i^{k_i})$,
and this implies \eqref{enit:gvichi}.
In the situation of \eqref{enit:gviphi}, $g_1(v)\neq 0$ and
by \eqref{enit:figv} we have $G_{v_i}\subseteq \ker(\chi_1^{k_1}\cdots
\chi_r^{k_r})$.
Applying Lemma \ref{lemma:newvanishingthm}\eqref{enit:gvkerchiorphi}
then shows \eqref{enit:gviphi}.
\end{proof}

\begin{remark}
Since $h$ is written in lowest terms in
Lemma \ref{lemma:vanishingthm}\eqref{enit:vanishing2},
nearly any choice of $v_i$ in $V(f_i)$ has $h_1(v_i)\neq 0$.
Thus, if $v_i$ is generic and
$f_i$ appears nontrivially in the denominator of an additive relative
invariant, then
by Lemma \ref{lemma:vanishingthm}\eqref{enit:figv},
Lemma \ref{lemma:vanishingthm}\eqref{enit:gvichi},
and Theorem \ref{thm:pvs1},
the subgroup $G^0_{v_i}$ 
lies in the kernel of every $\chi\in X_1(G)$.
\end{remark}

\begin{remark}
\label{rem:fisquared}
%
%
In the situation of Lemma \ref{lemma:vanishingthm}\eqref{enit:gvichi},
the results of 
\cite[\S5]{pike-generalizebrion}
imply that
the ideal $I$ of Remark \ref{rem:omegac}
is contained in $(f_i)^2$.
In particular, if $I\nsubseteq (f_i)^2$ then there are
no additive relative invariants with $f_i$ appearing in the denominator.
\end{remark}

\begin{remark}
\label{rem:dimsorbits}
Lemma \ref{lemma:vanishingthm} says much about the
dimensions of the orbits of
certain algebraic subgroups $H$ of $G$.
In particular, if $G^0_v\subseteq H$ then
$H^0_v=G^0_v$
and it follows that  
$\dim(H\cdot v)
=\dim(G\cdot v)-\codim(H)$.
\end{remark}

\subsection{Algebraic independence}
\label{subsec:algind}

By Theorem \ref{thm:pvs1},
a set of basic relative
invariants are algebraically independent polynomials;
in fact, we may add to this set the numerators of
a basis of the additive relative invariants.

\begin{proposition}
\label{prop:algind}
Let $f_1,\ldots,f_r$ be a set of basic relative invariants, with
$f_i\associatedmult \chi_i$.
Let $\Phi_1,\ldots,\Phi_s$ be a basis of $\A_1(G)$,
and for $1\leq i\leq s$
let $\frac{h_i}{g_i}$ be an additive relative invariant written in
lowest terms with 
$\frac{h_i}{g_i}\associatedadd \Phi_i$.
Let $v_i\in V(f_i)$ with $v_i\notin V(f_j)$ for $i\neq j$.
Then:
\begin{enumerate}
\item
\label{enit:fiinvariance}
Each $f_i$ is invariant under the action of $[G,G]\cdot G_{v_0}$,
and $G_{v_j}$ for $j\neq i$.
\item
\label{enit:hiinvariance}
Each $h_i$ is invariant under the action of $[G,G]\cdot G_{v_0}$,
and $G_{v_j}$ for $j$ with $g_i(v_j)\neq 0$.
\item
\label{enit:algebraicindependence}
The polynomials 
$f_1,\ldots,f_r,h_1,\ldots,h_s\in \C[V]$ are algebraically independent
over $\C$.
\end{enumerate}
\end{proposition}
\begin{proof}
Claim \eqref{enit:fiinvariance} follows from \eqref{eqn:relinv},
Proposition \ref{prop:x1is}, and
Lemma \ref{lemma:vanishingthm}\eqref{enit:figv}.
Claim \eqref{enit:hiinvariance} follows from
Propositions \ref{prop:globaleqn} and
\ref{prop:a1}, and
Lemma \ref{lemma:vanishingthm}\eqref{enit:gviphi}.

A criterion of Jacobi (e.g., \cite[\S3.10]{humphreys-reflection_groups})
states that a set of $m$ polynomials in $\C[V]$, $m\leq \dim(V)$,
is algebraically independent over $\C$ if and only if
some $m\times m$ minor of their Jacobian matrix is not the zero
polynomial.
We shall equivalently show that
the wedge product of their exterior derivatives is not identically
zero.

Let $\g$ be the Lie algebra of $G$, and for $X\in\g$ let
$\xi_X$ be the vector field on $V$ defined by
$\xi_X(v)=\frac{d}{dt}\left(\rho(\exp(tX))(v)\right)|_{t=0}$ (see also
\S\ref{subsec:lfdsintro}).
If $f\associatedmult \chi$ and $X\in \g$, then differentiating
\eqref{eqn:relinv} shows
\begin{equation}
\label{eqn:diffrelinv}
(\xi_X(f))(v)=df_{(v)}(\xi_X(v))=d\chi_{(e)}(X)\cdot f(v).
\end{equation}
Similarly, if
$X\in\g$ and $g_i\associatedmult\chi'_i$ then differentiating \eqref{eqn:allgv} shows
\begin{equation}
\label{eqn:diffaddinv}
d(h_i)_{(v)}(\xi_X(v))=d(\chi'_i)_{(e)}(X)\cdot h_i(v)+g_i(v)\cdot
d(\Phi_i)_{(e)}(X).
\end{equation}
Since $(\chi_1,\ldots,\chi_r,\Phi_1,\ldots,\Phi_s)(G)=
(\Gm)^r\times(\Ga)^s$, choose
$X_1,\ldots,X_r,Y_1,\ldots,Y_s\in\g$ such that
$d(\chi_i)_{(e)}(X_j)=\delta_{ij}$,
$d(\chi_i)_{(e)}(Y_j)=0$,
$d(\Phi_i)_{(e)}(X_j)=0$,
and
$d(\Phi_i)_{(e)}(Y_j)=\delta_{ij}$,
where $\delta_{ij}$ denotes the Kronecker delta function.
Because each $\chi'_i$ is some product of the $\chi_j$,
we have $d(\chi'_i)_{(e)}(Y_j)=0$.
%
Then by \eqref{eqn:diffrelinv} and \eqref{eqn:diffaddinv},
evaluating $df_1\wedge \cdots \wedge df_r\wedge dh_1\wedge\cdots\wedge
dh_s$ at $(\xi_{X_1},\cdots,\xi_{X_r},\xi_{Y_1},\cdots,\xi_{Y_s})(v)$
for any $v\in\Omega$
equals
%
\begin{equation*}
\begin{vmatrix}
\begin{smallmatrix}
f_1(v) & & \text{\large{0}} \\
 & \ddots & \\
\text{\large{0}} & & f_r(v)
\end{smallmatrix}
&
\begin{smallmatrix}
 \text{\Huge{0}} 
\end{smallmatrix}
\\
\begin{smallmatrix}
\phantom{g_1(v)} \\
\text{\Huge{*}} \\
\phantom{0}
\end{smallmatrix}
&
\begin{smallmatrix}
g_1(v) & & \text{\large{0}} \\
 & \ddots & \\
\text{\large{0}} & & g_r(v)
\end{smallmatrix}
\end{vmatrix}
\neq 0.\qedhere
\end{equation*}
\end{proof}

\begin{remark}
By Remark \ref{rem:dimsorbits} the dimension of the
$[G,G]\cdot G_{v_0}$--orbit of any $v\in\Omega$ 
is
$\dim(G\cdot v)-\codim([G,G]\cdot G_{v_0})
=\dim(V)-(r+s)$.
(Note that in $\Omega$ the orbits of
$[G,G]$ and $[G,G]\cdot G_{v_0}$ agree.)
Although 
the level set of
$(f_1,\ldots,f_r,h_1,\ldots,h_s)$
containing $v$
and the orbit of such a $v$ will
thus
agree locally,
it is doubtful that 
$f_1,\ldots,f_r,h_1,\ldots,h_s$ always generate the
subring of $\C[V]$ of all $[G,G]\cdot G_{v_0}$--invariant
polynomials,
as invariant subrings are not always
finitely generated.
\end{remark}

  \subsection{Decomposition}
\label{subsec:decomp}
For prehomogeneous vector spaces, any nontrivial
factor of a relative invariant is itself a relative invariant;
a crucial fact used in the proof is the unique factorization of a
rational function on $V$.

An analogous statement for additive relative
invariants
might be that any additive relative invariant may be expressed as a
sum of additive relative invariants with `simpler' denominators.
Example \ref{ex:denompowers} shows that these denominators may be
powers of basic relative invariants, and hence we conjecture:

\begin{conjecture}
\label{conj:canbreakup}
There exists a basis $\Phi_1,\ldots,\Phi_s$ for $\A_1(G)$ such that
if we choose reduced rational functions $h_i$ with 
$h_i\associatedadd\Phi_i$,
then for each $i$ the denominator of $h_i$ is a positive power of some
basic relative invariant.
\end{conjecture}

It would be natural to use a partial fraction expansion to prove
this,
but
such an expansion does not always
exist
for rational functions of several variables.
The question of the existence of a partial fraction decomposition
seems to be the main obstacle in proving this conjecture.

\begin{proposition}
\label{prop:canbreakup}
Let $\frac{h_1}{f_1f_2}$ be an additive relative invariant written in
lowest terms, with $f_1$ and $f_2$ polynomial relative invariants
having no common factors.
If
$$\frac{h_1}{f_1f_2}=\frac{\alpha}{f_1}+\frac{\beta}{f_2}$$
for $\alpha,\beta$ homogeneous polynomials of the same degree as
$f_1$ and $f_2$, respectively, then
$\frac{\alpha}{f_1}$ and $\frac{\beta}{f_2}$ are
additive relative invariants written in lowest terms.
\end{proposition}
\begin{proof}
The statement is true when $h_1=0$ and $f_1,f_2$ are constants, so
assume $h_1$ is nonzero.

Let $f_i\associatedmult\chi_i$ and $\frac{h_1}{f_1f_2}\associatedadd
\Phi$.
By algebra we have
\begin{equation}
\label{eqn:zstarstarstar}
h_1=\alpha\cdot f_2+\beta\cdot f_1
\end{equation}
and so by Proposition \ref{prop:globaleqn} we have for all $(g,v)\in
G\times V$,
\begin{multline}
\label{eqn:somehorrible}
\alpha(\rho(g)(v))\cdot f_2(\rho(g)(v))+\beta(\rho(g)(v))\cdot
f_1(\rho(g)(v)) \\
=\chi_1(g)\chi_2(g)\left(
  f_1(v) f_2(v) \Phi(g)+\alpha(v)f_2(v)+\beta(v)f_1(v)\right).
\end{multline}
Rearranging \eqref{eqn:somehorrible} and using \eqref{eqn:relinv}
gives
\begin{multline}
\label{eqn:somehorrible2}
\left(\alpha(\rho(g)(v))-\chi_1(g)\alpha(v)\right)\cdot \chi_2(g) f_2(v) \\
=\chi_1(g)f_1(v) \left(\chi_2(g) f_2(v)
\Phi(g)+\chi_2(g)\beta(v)-\beta(\rho(g)(v))\right).
\end{multline}
For each 
$g\in G$, the left side of \eqref{eqn:somehorrible2}
is divisible by $f_1$, and since $f_1$ and $f_2$ have no common
factors,
$\alpha\circ \rho(g)-\chi_1(g)\cdot \alpha$ is divisible by $f_1$.
By Proposition \ref{prop:funsagreeon},
$\frac{\alpha}{f_1}$ is an additive relative invariant,
and thus so
is
$\frac{\beta}{f_2}$.
If $\frac{\alpha}{f_1}$ or $\frac{\beta}{f_2}$ is not in lowest terms, then by
\eqref{eqn:zstarstarstar} neither is $\frac{h_1}{f_1f_2}$.
\end{proof}
\begin{remark}
If in Proposition \ref{prop:canbreakup} we omit the assumption that
$\alpha$ and $\beta$ are homogeneous, then we may take homogeneous
parts of \eqref{eqn:zstarstarstar}
to find 
an $\alpha'$ and $\beta'$ that are homogeneous of the correct degree.
\end{remark}

\section{Examples}
\label{sec:exs}
We now give some examples of additive relative invariants.
Our first two examples show there may be arbitrary numbers of
linearly independent additive relative invariants associated to a
single hypersurface, and arbitrary numbers of hypersurfaces.

\begin{example}
\label{ex:2addinvs}
Let
$
G=\left\{ \left(\begin{smallmatrix}
a & 0 & 0\\
b & a & 0 \\
c & 0 & a
\end{smallmatrix}\right)\in\GL_3(\C)\right\}$
act on $\C\{x,y,z\}=\C^3$ by multiplication, where this notation means
that $a,b,c$ may take any value that gives an invertible matrix.
Then $G$ has $D=V(x)$ as the exceptional orbit variety,
with $f_1=x\associatedmult \chi_1=a$.
Since $G$ is abelian and the isotropy subgroup at
$(x,y,z)=(1,0,0)\in\Omega $ is trivial,
$G$ has $2$ linearly independent additive functions,
namely, 
$h_1=\frac{y}{x}\associatedadd \Phi_1=\frac{b}{a}$
and
$h_2=\frac{z}{x}\associatedadd \Phi_2=\frac{c}{a}$.
On $V(x)$, radial subsets of the form $\alpha y+\beta z=0$
are invariant, and the $G$--action multiplies both coordinates by
$\chi_1=a$ as expected by Proposition \ref{prop:funsagreeon}.

More generally, fix coordinates $x_1,\ldots,x_n$ on $\C^n$ and let
$G\subseteq \GL(\C^n)$
consist of those matrices with all
diagonals equal, the first column unrestricted, and all other entries
equal to zero.
Then $G$ is abelian and has
$f_1=x_1$ and
$n-1$ linearly independent additive functions corresponding to
the additive relative invariants
$h_i=\frac{x_{i+1}}{x_1}$, $i=1,\ldots,n-1$.
\end{example}

\begin{example}
\label{ex:A}
Let
$
G=\left\{ \left(\begin{smallmatrix}
a & 0 & 0\\
b & a & 0 \\
0 & 0 & c
\end{smallmatrix}\right)\in\GL_3(\C)\right\}$
act on $\C\{x,y,z\}=\C^3$ by multiplication.
Then $G$ has 
$f_1=x\associatedmult \chi_1=a$
and
$f_2=z\associatedmult \chi_2=c$.
Since $G$ is abelian and the isotropy subgroup at
$(x,y,z)=(1,0,1)\in\Omega$ is trivial,
$\dim(\A_1(G))=1$;
a generator is
$h_1=\frac{y}{x}\associatedadd \Phi_1=\frac{b}{a}$.
On $V(x)$, $g\in G$ sends $(0,y,z)^T$ to $(0,\chi_1(g)
y,\chi_2(g)z)^T$; hence,
both
$V(x,y)$ and $V(x,z)$ are 
$G$--invariant, but only the former has the behavior predicted by
Proposition \ref{prop:funsagreeon}.
Both $V(x)$ and $V(z)$ contain open orbits.

More generally, we may take direct products of groups of the type in
Example \ref{ex:2addinvs}, realized as block diagonal matrices.
For $n\in\N$ and $k_1,\ldots,k_n\in \Z_{\geq 0}$,
this construction can produce a prehomogeneous vector space with
$\Omega^c$ consisting of
$n$ hyperplanes $H_1,\ldots,H_n$,
such that for each $i$ there are $k_i$ linearly independent additive
relative invariants having 
poles on $H_i$.
\end{example}

Next, we observe that the existence of an additive relative invariant
does not depend only on the orbit structure of the action. 

\begin{example}
Let
$
G=\left\{ \left(\begin{smallmatrix}
a & 0 & 0\\
b & \frac{a}{c} & 0 \\
0 & 0 & c
\end{smallmatrix}\right)\in\GL_3(\C)\right\}$
act on $\C\{x,y,z\}=\C^3$ by multiplication.
Then $G$ has 
$f_1=x\associatedmult \chi_1=a$
and
$f_2=z\associatedmult \chi_2=c$.
The generic isotropy subgroup is trivial and $\dim([G,G])=1$.
Thus $G/[G,G]\cdot G_{v_0}$ has no nontrivial additive functions, 
and there are no nontrivial additive relative invariants.
The orbit structure of $G$ agrees with that of the group in
Example \ref{ex:A},
but here neither $V(x,y)$ nor $V(x,z)$ exhibit the behavior
described in Proposition \ref{prop:funsagreeon}.
\end{example}

An additive relative invariant may unavoidably have a nontrivial power
of a basic relative invariant in its denominator.

\begin{example}
\label{ex:denompowers}
For $n\geq 2$, let $G\subseteq \GL_n(\C)$
consist of all invertible lower-triangular matrices $A$
such that each $(A)_{ij}$ depends only on $i-j$, that is, 
$(A)_{i+1,j+1}=(A)_{i,j}$ whenever this makes sense.
Then $G$ is a connected abelian linear algebraic group, and its
action on $\C^n=\C\{x_1,\ldots,x_n\}$ has an open orbit with
exceptional orbit variety
$V(x_1)$ and a trivial generic isotropy subgroup. 
It follows that $G$ has $n-1$ linearly independent additive functions.
For $1\leq i\leq n-1$ define $\Phi_i$ by
\begin{equation*}
\Phi_i\left(
 \begin{pmatrix}
a_1 & 0   & 0   & \cdots & 0 
\\
a_2 & a_1 & 0   & \ddots & \vdots
\\
a_3 & a_2 & a_1 & \ddots & \vdots
\\
\vdots & \ddots & \ddots & \ddots & \vdots
\\
a_n & a_{n-1} & a_{n-2} & \cdots & a_1
 \end{pmatrix}
\right)=
\frac{1}{(a_1)^i}
\begin{vmatrix}
\frac{1}{i}a_2 & a_1 & 0   & \cdots & 0 
\\
\frac{2}{i}a_3 & a_2 & a_1   & \ddots & \vdots
\\
\frac{3}{i}a_4 & a_3 & a_2 & \ddots & \vdots
\\
\vdots & \ddots & \ddots & \ddots & \vdots
\\
\frac{i}{i}a_{i+1} & a_{i} & \cdots & \cdots & a_2
\end{vmatrix}.
\end{equation*}
For instance, when $n=4$ we have
\begin{align*}
\Phi_1&=\frac{a_2}{a_1}, & 
\Phi_2&=\frac{1}{a_1^2}\left(\frac{1}{2} a_2^2 -a_1 a_3\right),
&
\text{and }
\Phi_3&=\frac{1}{a_1^3}\left(
 \frac{1}{3}(a_2)^3-a_1a_2a_3+a_1^2 a_4 \right).
\end{align*}
For at least $n\leq 20$, computer calculations show that
each $\Phi_i:G\to \Ga$ is a homomorphism, and hence
$\Phi_1,\ldots,\Phi_{n-1}$ is a basis of the space of additive
functions.
By \eqref{eqn:star},
the corresponding additive relative invariants are
of a form similar to $\Phi_i$ (e.g., substitute $x_i$ for $a_i$),
and for $i>1$ have nonreduced denominators.
%
%
%
%
\end{example}

\begin{example}
Fix a positive integer $n$ and let $m\in \{n,n+1\}$.
Let $L\subseteq \GL_n(\C)$ (respectively, $U\subseteq \GL_m(\C)$)
consist of invertible lower triangular matrices
(resp., upper triangular unipotent matrices).
Let $G=L\times U$ act on the space $M(n,m,\C)$ of $n\times m$ complex
matrices by $(A,B)\cdot M=AMB^{-1}$.
The classical \emph{LU factorization} of a complex
matrix asserts that this is a
prehomogeneous vector space (see \cite{DP-matrixsingI}).
For any matrix $M$, let $M^{(k)}$ denote the upper-leftmost $k\times
k$ submatrix of $M$.
By \cite[\S6]{DP-matrixsingI},
the basic relative invariants are of the form
$f_i(M)=\det(M^{(i)})$, $i=1,\ldots,n$,
and
$f_i\associatedmult \chi_i(A,B)=\det(A^{(i)})$.
%
For $I$ the identity matrix, $v_0=I$ when $m=n$, or
$v_0=\begin{pmatrix} I & 0 \end{pmatrix}$ when $m=n+1$,
is in $\Omega$ 
and has a trivial isotropy subgroup $G_{v_0}$.

Since $\dim(G/[G,G]\cdot G_{v_0})=n+m-1$, the quotient has
$m-1$
linearly independent additive functions.
In fact, for $(A,B)\in G$ and $1\leq i\leq m-1$, let
$\Phi_i(A,B)=-(B)_{i,i+1}$.
A computation shows that these $\Phi_i$ are linearly independent
additive
functions of $G$,
with
$$h_i(M)=\frac{\det((\textrm{$M$ with column $i$
deleted})^{(i)})}{\det(M^{(i)})}\associatedadd \Phi_i.$$
\end{example}

\section{Linear free divisors}
\label{sec:lfds}
We now consider prehomogeneous vector spaces
$(G,\rho,V)$ for which the complement of the open orbit
$\Omega$ is a type of hypersurface called a linear free divisor.
Our main theorem is that these have no nontrivial additive relative
invariants, but this has significant consequences for their structure.

  \subsection{Introduction}
\label{subsec:lfdsintro}
Let $\Ocnp$ denote the ring of germs of holomorphic functions on
$\C^n$ at $p$,
and $\Der_{\C^n,p}$ the $\Ocnp$--module of germs of holomorphic vector
fields on $\C^n$ at $p$.
Associated to a germ $(D,p)$ of a reduced analytic set in $\C^n$ is the
$\Ocnp$--module of \emph{logarithmic vector fields} defined by
$$
\Derlog{\C^n,p}{D}:=\{\eta\in \Der_{\C^n,p}: \eta(I(D))\subseteq I(D)\},$$
where $I(D)$ is the ideal of germs vanishing on $D$.
These are the vector fields tangent to $(D,p)$, and form
a Lie algebra using the Lie bracket of
vector fields.

Let $D$ be nonempty and $(D,p)\neq(\C^n,p)$.
When $\Derlog{\C^n,p}{D}$ is a free $\Ocnp$--module, necessarily of rank $n$,
then $(D,p)$ is called a \emph{free divisor}.
Free divisors are always pure hypersurface germs that are either smooth or
have singularities in codimension $1$,
and were first encountered as various types of discriminants.

Now let $(G,\rho,V)$ be a prehomogeneous vector space, let
$\g$ be the Lie algebra of $G$, and let $D:=V\setminus \Omega$
be the exceptional orbit variety.
Let $\Derlog{V}{D}_0$ denote the $\eta\in\Derlog{V}{D}$ that are
\emph{linear}, that is,
homogeneous of degree $0$ (e.g.,
$(x+y)\frac{\partial}{\partial x}-2z\frac{\partial}{\partial y}$).
Because $D$ is $G$--invariant, differentiating the action of $G$ gives a Lie algebra
\mbox{(anti-)}homomorphism
$\tau:\g\to\Derlog{V}{D}_0$, where
$\xi_X:=\tau(X)$ is a vector field defined globally on $V$ by
$$\xi_X(v)=\tau(X)(v)=\frac{d}{dt}\!\left.\big(\rho(\exp(tX))(v)\big)\right|_{t=0}.$$
Thus, $\tau(\g)\subseteq \Derlog{V}{D}_0$ are
finite-dimensional Lie subalgebras of the module $\Derlog{V}{D}$. 
As $\tau(\g)(v)=T_v(G\cdot v)$, the maximal minors of a matrix
containing 
the coefficients of a basis of $\tau(\g)$ are of degree $\dim(V)$ and
generate an ideal defining the set $D$.

When $\Derlog{V}{D}$ has a free basis of linear vector fields,
then $D$ is called a \emph{linear free divisor}.
By
\cite[Lemma 2.3]{gmns},
every linear free divisor $D$ is the exceptional orbit variety of
a prehomogeneous vector space
with the following properties:
\begin{definition}
\label{def:defineslfd}
Let $D$ be a linear free divisor in $V$.
If $(G,\rho,V)$ is a prehomogeneous vector space with
open orbit $V\setminus D$, 
$G$ connected,
and $\dim(G)=\dim(V)$,
then say that
\emph{$G$ defines the linear free divisor $D$}.
\end{definition}
It follows that
$\ker(\rho)$ is finite, and
$\tau(\g)=\Derlog{V}{D}_0$.

One such prehomogeneous vector space 
may be constructed in the following way. 
For a divisor $D\subset V$ in a finite-dimensional complex vector
space,
let $\GL(V)_D$ be the largest subgroup
of $\GL(V)$ that preserves $D$.
Note that $\GL(V)_D$ is algebraic.
Then $(\GL(V)_D)^0$ is a connected complex linear algebraic group
with Lie algebra (anti-)isomorphic to 
$\Derlog{V}{D}_0$.
For a linear free divisor $D\subset V$, the group
$(\GL(V)_D)^0\subseteq \GL(V)$
with the inclusion representation is a prehomogeneous vector space that
defines $D$
(\cite[Lemma 2.3]{gmns}).
In fact, 
if $(G,\rho,V)$ defines a
linear free divisor $D$, then $\rho(G)=(\GL(V)_D)^0$.

  \subsection{Brion's criterion}
\label{subsec:brion}

Brion gave the following useful criterion for $D$ to be a linear free
divisor.
\begin{theorem}[\cite{brion}, {\cite[Theorem 2.1]{freedivisorsinpvs}};
see also \cite{pike-generalizebrion}]
\label{thm:brion}
For $G=(\GL(V)_D)^0$,
the following are equivalent:
\begin{enumerate}
\item
\label{cond:dlfd}
$D$ is a linear free divisor.
\item
\label{cond:bothopenisotrpy}
Both of these conditions hold:
\begin{enumerate}
 \item
\label{cond:vmdopen}
$V\setminus D$ is a unique $G$-orbit, and the corresponding
isotropy groups are finite.
 \item
\label{cond:isotropyrep}
Each irreducible component $D_i$ of $D$ contains an open
$G$-orbit $D_i^0$, and the corresponding isotropy groups are
extensions of finite groups by $\Gm$.
\end{enumerate}
\end{enumerate}
When these hold,
$\tau(\g)$ 
generates $\Derlog{V}{D}$, and each $D_i^0=\smooth(D)\cap D_i$.
\end{theorem}

The proof of Theorem \ref{thm:brion} shows the following.

\begin{corollary}
\label{cor:tobrion}
Suppose that in the situation of Theorem \ref{thm:brion},
\eqref{cond:vmdopen} holds,
$D_i$ contains an open orbit $D_i^0$, and $v_i\in D_i^0$.
Then $G_{v_i}$ is an extension of a finite group by $\Gm$
if and only if
the induced representation of $G^0_{v_i}$ on the normal line to $D_i$ at
$v_i$ is nontrivial.
\end{corollary}

The representation on the normal line is actually quite familiar.

\begin{lemma}
\label{lemma:interpnormal}
Let $(G,\rho,V)$ be a prehomogeneous vector space with $f$ as a basic
relative invariant, and let $v$ be a smooth point of $D:=V(f)$.
If $H\subseteq G_v$, 
then the representation
$\rho_v:H\to\GL(L)$ on the
normal line $L=T_v V/T_vD$ to $D$ at $v$ is
\begin{equation*}
\rho_v(h)(\ell)=\chi(h)\ell,\qquad \textrm{for all }\ell\in L,
\end{equation*}
where $f\associatedmult\chi$.
\end{lemma}

Geometrically,
$\rho_v$ acts on a normal slice to $f=0$ at $v$, and such a slice
intersects all level sets of $f$.
At the same time, the action of $\rho(H)$ fixes $v$ and
translates between the level sets of $f$ according to $\chi$.

%
\begin{proof}
The representation $\rho|_H$ fixes $v$ and hence
induces a representation $\rho'$ of $H$
on the tangent space $T_vV$;
by silently identifying $T_vV$ with $V$, we have $\rho'=\rho|_H$.
Since $\rho|_H$ leaves invariant $D$ and fixes the smooth point $v$ on
$D$,
the representation $\rho'$ leaves invariant $T_vD$.
Then the normal line is $L=T_vV/T_vD$, and $\rho'$ produces the quotient
representation $\rho_v:H\to \GL(L)$ defined by
$\rho_v(h)(w+T_vD)=\rho'(h)(w)+T_vD$.

Let $h\in H$ and $w\in T_vV$.
For $\lambda\in\C$ we have by \eqref{eqn:relinv},
\begin{equation}
\label{eqn:fvlambda}
f(v+\lambda\cdot \rho(h)(w))=f(\rho(h)(v+\lambda w))=\chi(h)
f(v+\lambda w).
\end{equation}
Differentiating \eqref{eqn:fvlambda} with respect to $\lambda$ and
evaluating at $\lambda=0$ gives
$df_{(v)}(\rho'(h)(w))=\chi(h) df_{(v)}(w)$,
or
$\rho'(h)(w)-\chi(h)w\in\ker(df_{(v)})$.
Since $f$ is reduced and $v$ is a
smooth point, $T_vD=\ker(df_{(v)}:T_vV\to T_0\C)$.
The definition of $\rho_v$ then implies the result. 
\end{proof}

\begin{remark}
If $(G,\rho,V)$ defines a linear free divisor $D$
in the sense of Definition \ref{def:defineslfd}, then
all of the results of this \S\ref{subsec:brion} hold for
$G$ as well as
$\rho(G)=(\GL(V)_D)^0$.
\end{remark}

  \subsection{The main theorems}
\label{subsec:mainthm}
Let $(G,\rho,V)$ define the linear free divisor
$D\subset V$ in the sense of Definition \ref{def:defineslfd}.
Let $f_1,\ldots,f_r$ be the basic relative invariants, so that
$\cup_{i=1}^r V(f_i)$ is the irreducible decomposition of
$D=V\setminus \Omega$.
Let $f_i\associatedmult \chi_i$, and choose $v_0\in\Omega$.
For $1\leq i\leq r$ let $v_i$ be a generic point on $V(f_i)$,
an element of the open orbit of $G$ in $V(f_i)$.

\begin{theorem}
\label{thm:homomorphism}
Let 
$G$ define the linear free divisor $D$, with the notation above.
Then:
\begin{enumerate}
\item
\label{enit:1}
The homomorphism $(\chi_1,\ldots,\chi_r):G\to (\Gm)^r$
is surjective and has kernel
$[G,G]\cdot G_{v_0}$.
\item
\label{enit:2}
$G/[G,G]\cdot G_{v_0}$ has only the trivial additive function and
$(G,\rho,V)$ has only constant additive relative invariants.
\item
\label{enit:3}
For $i\neq j$, we have $G^0_{v_j}\subseteq \ker(\chi_i)$.
\item
\label{enit:4}
The representation $\chi_i|_{G^0_{v_i}}:G^0_{v_i}\to\Gm$ is surjective, with finite
kernel.
\item
\label{enit:5}
$\ker(\chi_i|_{G^0_{v_i}})=G^0_{v_i}\cap ([G,G]\cdot G_{v_0})$ is
finite.
\item
\label{enit:6}
For $i\neq j$, the subgroup $G^0_{v_i}\cap G^0_{v_j}$ is a finite subset of
$[G,G]\cdot G_{v_0}$.
\end{enumerate}
\end{theorem}

\begin{proof}
Let $G_1=[G,G]\cdot G_{v_0}$.
Claim 
\eqref{enit:3} is just  
Lemma \ref{lemma:vanishingthm}\eqref{enit:figv}.

By
Theorem \ref{thm:brion},
Corollary \ref{cor:tobrion}, and Lemma \ref{lemma:interpnormal},
each $\chi_i|_{G^0_{v_i}}$ is nontrivial, thus surjective.
As $\dim(G^0_{v_i})=1=\dim(\GL(\C))$, the kernel is finite, giving
\eqref{enit:4}.

If $\Phi\in\A(G/G_1)$ is nontrivial, then by Proposition \ref{prop:a1} there
exists a non-constant $h$ with $h\associatedadd \Phi$.
By Proposition \ref{prop:homogeneous}, write
$h=\frac{h_1}{g_1}$ as a reduced fraction, with
$g_1$ a polynomial relative invariant and
$g_1\associatedmult \chi$.
Since $\Phi$ is nontrivial, $g_1$ cannot be constant and so by
Theorem \ref{thm:pvs1} is a product of basic relative invariants;
let $f_i$ be an irreducible factor of $g_1$.
If $h_1(v_i)=0$, then by Proposition \ref{prop:globaleqn},
$h_1$ vanishes on $G\cdot v_i$ and hence on
$\overline{G\cdot v_i}=V(f_i)$; but then $f_i$ divides
$\gcd(h_1,g_1)$, a contradiction of how we wrote $h$, hence $h_1(v_i)\neq 0$.
By
Lemma \ref{lemma:vanishingthm}\eqref{enit:gvichi}
we have $G^0_{v_i}\subset \ker(\chi_i)$,
but this contradicts \eqref{enit:4}.
Thus $\Phi=0$ is trivial, proving 
\eqref{enit:2}.

Since $G/G_1$ is connected and abelian,
by Proposition \ref{prop:ckcl},
$G/G_1\cong (\Gm)^k\times(\Ga)^\ell$.
Choosing a free basis $\epsilon_1,\ldots,\epsilon_k$ of $X(G/G_1)$
and a basis $\Phi_1,\ldots,\Phi_\ell$ of $\A(G/G_1)$
gives an explicit isomorphism
$\theta=(\epsilon_1,\ldots,\epsilon_k,\Phi_1,\ldots,\Phi_\ell):G/G_1\to
(\Gm)^k\times (\Ga)^\ell$.
Theorem \ref{thm:pvs1} and Proposition \ref{prop:x1is} show that
we may let each $\epsilon_i=\chi_i$.
Then, \eqref{enit:1} follows from \eqref{enit:2}.

For \eqref{enit:5},
such elements are contained in the kernel by \eqref{enit:1},
and conversely by \eqref{enit:3} and \eqref{enit:1}.
By \eqref{enit:4}, the kernel is finite.

For \eqref{enit:6},
observe that by \eqref{enit:3}, such elements are in the kernel of the
homomorphism of \eqref{enit:1}.
For finiteness, use \eqref{enit:5}.
\end{proof}

\begin{remark}
For each $i>0$, the group $G^0_{v_i}$ is a  
$1$--dimensional connected complex linear algebraic group with
a surjective homomorphism to $\Gm$,
and hence 
$G^0_{v_i}\cong \Gm$.
\end{remark}

Theorem \ref{thm:homomorphism} has a number of immediate consequences.
In particular,
we may easily compute the number of irreducible components.

\begin{theorem}
\label{thm:ofhomomorphism}
Let $G$ define the linear free divisor $D$, with the notation above.
Then:
\begin{enumerate}
\item
\label{enit:unique1}
The hypersurface $D$ has 
$$r=\dim_\C(G/[G,G]\cdot
G_{v_0})=\dim_\C(G/[G,G])=\dim_\C(\g/[\g,\g])$$
irreducible
components, where $\g$ is the Lie algebra of $G$.
\item
\label{enit:unique2}
Every element of $G$ may be written in a finite number of ways
as a product of elements
from the subgroups $[G,G]\cdot G_{v_0},G^0_{v_1},\ldots,G^0_{v_r}$.
Each term of such a product is unique modulo
$[G,G]\cdot G_{v_0}$.
\item
\label{enit:unique3}
Let $S\subseteq \{1,\ldots,r\}$.
The subgroup of $G$ that leaves invariant all level sets of $f_i$
for $i\in S$ is normal in $G$, and is the product of the subgroups
$[G,G]\cdot G_{v_0}$,
and $G^0_{v_j}$ for $j\notin S$.
\item
\label{enit:unique5}
As algebraic groups, $G/[G,G]\cong (\Gm)^r$.
\item
\label{enit:unique4}
The subgroup $[G,G]$ contains all unipotent elements of $G$.
\end{enumerate}
\end{theorem}
\begin{proof}
For \eqref{enit:unique1}, combine
Theorem \ref{thm:homomorphism}\eqref{enit:2}
and Corollary \ref{cor:numcomponents}.
Note also that $G_{v_0}$ is finite by
Theorem \ref{thm:brion}.

Let $g\in G$.
By Theorem \ref{thm:homomorphism}\eqref{enit:4},
for $i>0$ there exists 
$g_i\in G^0_{v_i}$ such that $\chi_i(g)=\chi_i(g_i)$.
By Theorem \ref{thm:homomorphism}\eqref{enit:3} and 
\ref{thm:homomorphism}\eqref{enit:1},
$g(g_1\cdots g_r)^{-1}$ is in the kernel of each $\chi_i$, and
hence lies in $[G,G]\cdot G_{v_0}$.
This proves existence for \eqref{enit:unique2}.

To address uniqueness, let $g\in G$ and suppose 
that for $j=1,2$ we have
$g=g_{0,j}\cdot g_{1,j}\cdots g_{r,j}$,
with
$g_{0,j}\in [G,G]\cdot G_{v_0}$
and
$g_{i,j}\in G^0_{v_i}$ for $i>0$. 
By Theorem \ref{thm:homomorphism}\eqref{enit:1} and
\ref{thm:homomorphism}\eqref{enit:3},
for $i>0$ we have $\chi_i(g)=\chi_i(g_{i,1})=\chi_i(g_{i,2})$.
Then for $i>0$ we have 
$g_{i,1}=g_{i,2}$ modulo $\ker(\chi_i|_{G^0_{v_i}})$,
and this kernel is finite and 
contained in $[G,G]\cdot G_{v_0}$ by Theorem
\ref{thm:homomorphism}\eqref{enit:5}.
Since $g_{0,j}$ is uniquely determined by $g_{1,j},\ldots,g_{r,j}$,
there are precisely
$$
\prod_{i=1}^r \#(\ker(\chi_i|_{G^0_{v_i}}))$$
ways to write $g$ in this way.  This proves \eqref{enit:unique2}.

By \eqref{eqn:relinv}, $g\in G$ leaves invariant all level sets of
$f_i$ if and only if $g\in\ker(\chi_i)$.
Hence, the subset $H$ leaving invariant all level sets of $f_i$ for
$i\in S$ is an intersection of kernels and thus a normal subgroup.
By Theorems \ref{thm:homomorphism}\eqref{enit:1} and 
\ref{thm:homomorphism}\eqref{enit:3},
$[G,G]\cdot G_{v_0}$ and $G^0_{v_j}$ for $j\notin S$ are in $H$.
Conversely, by
\eqref{enit:unique2} and Theorems
\ref{thm:homomorphism}\eqref{enit:3} and
\ref{thm:homomorphism}\eqref{enit:5}, and the normality of
$[G,G]\cdot G_{v_0}$, any element of $H$ may be written as a product
of elements of these subgroups.
This proves \eqref{enit:unique3}.

To prove \eqref{enit:unique5},
consider the diagram
\begin{equation*}
\xymatrix@-1pc{
G/[G,G] \ar[r]^-\phi \ar[d]^{\kappa}
  & \Gm^k\times\Ga^\ell \ar[d]^{\psi} \\
G/[G,G]\cdot G_{v_0} \ar[r]^-\chi
  & \Gm^r
}
\end{equation*}
where
$\phi$ is an isomorphism that exists by Proposition \ref{prop:ckcl},
$\kappa$ is the quotient map,
$\chi$ is the isomorphism induced from $(\chi_1,\ldots,\chi_r)$
by Theorem \ref{thm:homomorphism}\eqref{enit:1},
and $\psi$ makes the diagram commutative.
The Jordan decomposition implies that $\{1\}\times \Ga^\ell\subseteq
\ker(\psi)$,
but this kernel is isomorphic to $\ker(\kappa)$, which is finite
because $G_{v_0}$ is finite.
Hence $\ell=0$ and $k\leq r$, and
the surjectivity of $\psi$ requires $k\geq r$,
proving \eqref{enit:unique5}.

Finally, \eqref{enit:unique4} is a consequence of \eqref{enit:unique5}
and the Jordan decomposition.
\end{proof}

\begin{example}[{\cite[Example 2.1]{mondbuchweitz}},{\cite[Example
5.1]{gmns}}]
\label{ex:aac}
On $\C^3$, fix coordinates $x,y,z$.
The linear free divisor 
$x(xz-y^2)=0$ is defined by the solvable 
group
$$G=\left\{\left(\begin{smallmatrix}
a & 0 & 0 \\ 
b & c & 0 \\
\frac{b^2}{a} & \frac{2bc}{a} & \frac{c^2}{a}
\end{smallmatrix}\right)\in
\GL(\C^3)\right\}.$$
We have
$f_1=x\associatedmult \chi_1=a$ and
$f_2=xz-y^2\associatedmult \chi_2=c^2$.
At $v_0=(1,0,1)\in \Omega$,
$v_1=(0,1,0)\in V(f_1)$,
and $v_2=(1,0,0)\in V(f_2)$,
the (generic) 
isotropy subgroups are defined by,
respectively,
$(a,b,c)=(1,0,\pm 1)$,
$(b,c)=(0,1)$,
and
$(a,b)=(1,0)$.
As $[G,G]$ is defined by $a=c=1$,
we see $[G,G]\cdot G_{v_0}$ is defined by $a=1$, $c=\pm 1$.
Theorems \ref{thm:homomorphism} and \ref{thm:ofhomomorphism}
are easy to verify.
\end{example}

Recall that $L(K)$ denotes the Lie algebra of an algebraic group $K$.
Let $\delta_{ij}$ denote the Kronecker delta function.
On the level of Lie algebras,
Theorem \ref{thm:homomorphism} implies
the following.

\begin{corollary}
\label{cor:homomorphismliealgebra}
As vector spaces,
$$\g=[\g,\g]\oplus \bigoplus_{i=1}^r L(G_{v_i}).$$
For $i=1,\ldots,r$ 
there exist unique $X_i\in L(G_{v_i})$
such that $L(G_{v_i})=\C X_i$,
and such that for all $j$ we have
$d(\chi_i)_{(e)}(X_j)=\delta_{ij}$
and 
$\xi_{X_j}(f_i)=\delta_{ij}\cdot f_i$.
For $X\in [\g,\g]$ and any $j$, $\xi_X(f_j)=0$.
\end{corollary}
\begin{proof}
Differentiating the homomorphism of Theorem
\ref{thm:homomorphism}\eqref{enit:1}
gives a homomorphism $\lambda:\g\to L((\Gm)^r)=\oplus_{i=1}^r L(\Gm)$ with kernel
$[\g,\g]$.
By Theorems \ref{thm:homomorphism}\eqref{enit:3} and
\ref{thm:homomorphism}\eqref{enit:4},
under $\lambda$ each $L(G_{v_i})$ surjects onto the $i$th copy of
$L(\Gm)$, and is zero on the rest.
This gives the vector space decomposition, and proves that for
$i\neq j$, $d(\chi_i)_{(e)}(L(G_{v_j}))=0$.
Now choose the unique $X_i\in L(G_{v_i})$ such that $d(\chi_i)_{(e)}(X_i)=1$.

The rest of the statement follows from
\eqref{eqn:diffrelinv}.
\end{proof}

Each $X_i$ depends on the choice of $v_i\in V(f_i)$,
but
any two choices will differ by an element of
$[\g,\g]$. 

\begin{remark}
\label{rem:worksmoregenerally}
More generally, let $(G,\rho,V)$ be a prehomogeneous vector space with
no nontrivial additive relative invariants.
(For instance, by Remark \ref{rem:fisquared}, this happens
if the ideal of Remark \ref{rem:omegac}
is not contained in $(f_i)^2$ for every basic relative invariant $f_i$.)
Then by the same argument in Theorems \ref{thm:homomorphism} and
\ref{thm:ofhomomorphism},
$G/[G,G]\cdot G_{v_0}$ is an
algebraic torus, of dimension equal to the number of irreducible hypersurface
components of the exceptional orbit variety $V\setminus \Omega$.
\end{remark}

\subsection{The structure of \texorpdfstring{$G$}{G}}
\label{subsec:structureg}
We now use Theorems \ref{thm:homomorphism} and 
\ref{thm:ofhomomorphism} to study the structure of algebraic groups
defining linear free divisors.

Let $G$ be a connected complex algebraic group.
Let $\Rad(G)$ denote the \emph{radical} of $G$, the maximal
connected normal solvable subgroup.
The (algebraic) \emph{Levi decomposition} of $G$
writes 
$$G=\Rad_u(G)\rtimes L,$$
where $\Rad_u(G)$ is the \emph{unipotent radical} of
$G$,
the largest connected unipotent normal subgroup of $G$, consisting of
all unipotent elements of $\Rad(G)$; and $L$ is
a \emph{Levi subgroup}, a maximal connected reductive algebraic subgroup of $G$,
unique up to conjugation
(\cite[11.22]{borel}). 
Moreover, $L=\gpcenter(L)^0\cdot [L,L]$
for 
$\gpcenter(L)$ the center of $L$,
$\gpcenter(L)\cap [L,L]$ is finite,
$[L,L]$ is semisimple,
and $(\gpcenter(L))^0=L\cap \Rad(G)$ is a maximal torus of
$\Rad(G)=\Rad_u(G)\rtimes \gpcenter(L)^0$
(\cite[14.2, 11.23]{borel}). 

Groups defining linear free divisors have the following structure.

\begin{corollary}
\label{cor:ofhomomorphism2}
Let $G$ define the linear free divisor $D$ and have the
Levi decomposition
above.  Then:
\begin{enumerate}
\item
\label{enit:gencomponents}
The number of irreducible components of $D$ equals 
$\dim(\gpcenter(L))$.
\item
\label{enit:genbracket}
$[G,G]=\Rad_u(G)\rtimes [L,L]$.
\item
\label{enit:geng}
$G=[G,G]\cdot (\gpcenter(L))^0$, with $[G,G]\cap (\gpcenter(L))^0$ finite.
\item
\label{enit:genisomorphism}
$(\chi_1,\ldots,\chi_r)|_{\gpcenter(L)^0}:\gpcenter(L)^0\to (\Gm)^r$ is
surjective, with a finite kernel.
\end{enumerate}
\end{corollary}
\begin{proof}
Let $R=\Rad_u(G)$.
Since $G$ is the semidirect product of $R$ and $L$,
a straightforward calculation shows that
$[G,G]=[R,R]\cdot [R,L]\cdot [L,L]$.
%
%
%
%
%
By Theorem \ref{thm:ofhomomorphism}\eqref{enit:unique4} and connectedness,
$R\subseteq [G,G]$.
Since $R\nsub G$, we have $[R,R],[R,L]\subseteq R$ and hence
$$
[G,G]
= R\cdot [G,G] 
= R\cdot [R,R]\cdot [R,L]\cdot [L,L] 
= R\cdot [L,L].$$
As $R$ consists of unipotent elements and 
$\Rad(G)\cap L$ consists of semisimple elements,
$R\cap L=\{e\}$ and hence $R\cap [L,L]=\{e\}$.
Since $R\subseteq [G,G]$ are normal subgroups of $G$,
we have $R\nsub [G,G]$.
This proves the decomposition \eqref{enit:genbracket}
as abstract groups, and hence as complex algebraic groups
(\cite[\S3.3--3.4]{ov}).

From this, we conclude 
\begin{equation*}
\label{eqn:gstructure}
G=R\cdot L=R\cdot [L,L]\cdot (\gpcenter(L))^0=[G,G]\cdot (\gpcenter(L))^0.
\end{equation*}
Clearly, $[L,L]\subseteq [G,G]\cap L$.
If $g\in[G,G]\cap L$, then by \eqref{enit:genbracket} we have
$g=k\cdot \ell$ for $k\in R$ and $\ell\in [L,L]\subseteq L$, hence
$k=g\ell^{-1}\in L\cap R=\{e\}$ and so $g=\ell\in[L,L]$; thus $[G,G]\cap L=[L,L]$.
In particular, $[G,G]\cap (\gpcenter(L))^0=
[L,L]\cap (\gpcenter(L))^0$, which is finite.
This proves \eqref{enit:geng}.

By \eqref{enit:geng} and an isomorphism theorem,
$$
G/[G,G]
\cong (\gpcenter(L))^0/([G,G]\cap \gpcenter(L)^0),$$
and hence $\dim(\gpcenter(L))=\dim(G/[G,G])$.
Theorem \ref{thm:ofhomomorphism}\eqref{enit:unique1}
then proves \eqref{enit:gencomponents}.

Finally, \eqref{enit:genisomorphism}
may be checked on the level of Lie algebras.
By \eqref{enit:geng}, we have
$\g=[\g,\g]\oplus \z(\fl)$, where
$\g$ and $\fl$ are the Lie algebras of $G$ and $L$,
and $\z(\fl)$ is the center of $\fl$.
For $\chi=(\chi_1,\ldots,\chi_r)$,  
we have
$\ker(d\chi_{(e)})=[\g,\g]$
by Theorem \ref{thm:homomorphism}\eqref{enit:1},
and hence $d\chi_{(e)}|_{\z(\fl)}$ is
an isomorphism.
In fact, by Theorem \ref{thm:homomorphism}\eqref{enit:1}
the kernel in \eqref{enit:genisomorphism} is
$(\gpcenter(L))^0\cap ([G,G]\cdot G_{v_0})$.
\end{proof}

\begin{remark}
An arbitrary connected complex linear algebraic group with Levi
decomposition $G=\Rad_u(G)\rtimes L$ has
$$[G,G]=
[\Rad_u(G),\Rad_u(G)]\cdot [\Rad_u(G),L]\cdot [L,L]
\subseteq \Rad_u(G)\rtimes [L,L],$$
and hence
$[\Rad_u(G),\Rad_u(G)]\cdot [\Rad_u(G),L]\subseteq \Rad_u(G)$;
if $G$ defines a linear free divisor, then 
by
Corollary \ref{cor:ofhomomorphism2}\eqref{enit:genbracket}
these are both equalities.
\end{remark}

\begin{remark}
\label{rem:boreltori}
If $B$ is a \emph{Borel subgroup} of $[L,L]$, a maximal connected
solvable subgroup, and $T\subseteq B$ is a maximal
torus of $[L,L]$, then it follows from \cite[11.14]{borel}
that $\Rad_u(G)\cdot B\cdot \gpcenter(L)^0$ is a Borel subgroup of $G$
and $T\cdot \gpcenter(L)^0$ is a maximal torus of $G$.
In particular, the number of irreducible components is at most
the dimension of the maximal torus.
\end{remark}

Consider the following examples of linear free divisors.

\begin{example}
\label{ex:aacpart2}
We continue Example \ref{ex:aac}.
Since $G$ is solvable, $[L,L]$ is trivial and hence
$\Rad_u(G)$ is defined by $a=c=1$,
and a maximal torus $L=\gpcenter(L)$ is defined by $b=0$.
Corollary \ref{cor:ofhomomorphism2} is easy to check;
in particular, the $2$-element subgroup
$L\cap G_{v_0}$ lies in the kernel of
$(\chi_1,\chi_2):G\to\Gm^2$ restricted to $(\gpcenter(L))^0$.
\end{example}

The following example is neither reductive nor solvable.

\begin{example}[{\cite[Example 9.4]{DP-matrixsingI}}]
\label{ex:complicated}
Define the algebraic group 
$$
G=\left\{\left(\begin{smallmatrix}
a & 0 & 0 & 0 \\
0 & b & c & 0 \\
0 & d & e & 0 \\
f & g & h & i\end{smallmatrix}\right)\in\GL(\C^4)\right\}.$$
Let $S$ be the space of $4\times 4$ symmetric matrices with
the usual coordinates
$x_{ij}$, $1\leq i\leq j\leq 4$. 
Let $V\subset S$ be the subspace where $x_{11}=0$.
Let $\rho:G\to\GL(V)$ be defined by $\rho(A)(M)=AMA^T$.
Note that $\ker(\rho)=\{\pm I\}$.
A Levi decomposition of $G$ has
$L$ defined by $f=g=h=0$
and $\Rad(G)$ defined by $c=d=b-e=0$
(the Borel subgroup of lower-triangular $g\in G$ is not normal in $G$).
Then
$\Rad_u(G)$ is defined by $a=b=e=i=1$ and $c=d=0$,
$[L,L]$ is defined by $f=g=h=0$ and $be-cd=a=i=1$,
and
$L\cap \Rad(G)=\gpcenter(L)^0$ is defined by $b=e$ and $c=d=f=g=h=0$. 
Finally,
$[G,G]$ is defined by
$a=be-cd=i=1$.

The exceptional orbit variety is the linear free divisor
defined by $f_1\cdot f_2\cdot f_3$,
where
$$
f_1=\begin{vmatrix}
x_{22} & x_{23} \\
x_{23} & x_{33}
\end{vmatrix},
\qquad
f_2=
\begin{vmatrix}
0      & x_{12} & x_{13} \\
x_{12} & x_{22} & x_{23} \\
x_{13} & x_{23} & x_{33} \end{vmatrix},
\qquad
f_3=
\begin{vmatrix}
0      & x_{12} & x_{13} & x_{14} \\
x_{12} & x_{22} & x_{23} & x_{24} \\
x_{13} & x_{23} & x_{33} & x_{34} \\
x_{14} & x_{24} & x_{34} & x_{44} \end{vmatrix},
$$
corresponding to the characters
$\chi_1=(cd-be)^2$,
$\chi_2=a^2(cd-be)^2$,
$\chi_3=a^2(cd-be)^2i^2$, respectively.
Let
$$
v_0=
\left(\begin{smallmatrix}
0 & 1 & 0 & 0 \\
1 & 1 & 0 & 0 \\
0 & 0 & 1 & 0 \\
0 & 0 & 0 & 1 \end{smallmatrix}\right),
\qquad
v_1=
\left(\begin{smallmatrix}
0 & 1 & 0 & 0 \\
1 & 0 & 0 & 0 \\
0 & 0 & 1 & 0 \\
0 & 0 & 0 & 1 \end{smallmatrix}\right),
\qquad
v_2=
\left(\begin{smallmatrix}
0 & 0 & 1 & 1 \\
0 & 0 & 1 & 0 \\
1 & 1 & 0 & 0 \\
1 & 0 & 0 & 0 \end{smallmatrix}\right),
\qquad
v_3=
\left(\begin{smallmatrix}
0 & 1 & 0 & 0 \\
1 & 1 & 0 & 0 \\
0 & 0 & 1 & 0 \\
0 & 0 & 0 & 0 \end{smallmatrix}\right)
\in V
$$
be generic points in $\Omega$ and on each $V(f_i)$.
Then $G_{v_0}$, $G_{v_1}$, $G_{v_2}$, $G_{v_3}$ respectively
consist of all elements of $G$ of the form
$$
\left(
 \begin{smallmatrix}
 \gamma & 0 & 0 & 0 \\
 0 & \gamma & 0 & 0 \\
 0 & 0 & \delta & 0 \\
 0 & 0 & 0 & \epsilon 
 \end{smallmatrix}\right),
\qquad
\left(
 \begin{smallmatrix}
 \frac{1}{a} & 0 & 0 & 0 \\
 0 & a & 0 & 0 \\
 0 & 0 & \delta & 0 \\
 0 & 0 & 0 & \epsilon 
 \end{smallmatrix}\right),
\qquad
\left(
 \begin{smallmatrix}
 \frac{1}{a} & 0 & 0 & 0 \\
 0 & \frac{1}{a} & 0 & 0 \\
 0 & 0 & a & 0 \\
 0 & 0 & 0 & a 
 \end{smallmatrix}\right),
\qquad
\left(
 \begin{smallmatrix}
 \epsilon & 0 & 0 & 0 \\
 0 & \epsilon & 0 & 0 \\
 0 & 0 & \delta & 0 \\
 0 & 0 & 0 & a 
 \end{smallmatrix}\right),
$$
where $\gamma^2=\delta^2=\epsilon^2=1$ and $a\in\C$.
With these calculations,
it is 
straightforward to check
the conclusions of
Theorems \ref{thm:homomorphism}
and \ref{thm:ofhomomorphism} and Corollary
\ref{cor:ofhomomorphism2}.
\end{example}

\subsection{The structure of the Lie algebra}
\label{subsec:liealgebra}
We now summarize our results for the structure of the Lie algebra $\g$ of $G$.
In contrast to the terminology for groups,
the usual \emph{Levi decomposition} of $\g$
expresses $\g$ as the semidirect sum of the \emph{radical}
$\radg$,
defined as
the maximal solvable ideal,
and a semisimple \emph{Levi subalgebra};
these correspond to the Lie algebras
of $\Rad(G)$ and $[L,L]$, respectively.

\begin{proposition}
Let $(G,\rho,V)$ define a linear free divisor $D\subset V$.
Let $G$ have a Levi decomposition as above, and let
$\g$, $\radu$, and $\lfrak$ be the Lie algebras of $G$, $\Rad_u(G)$, and
$L$, respectively, and let $\z(\lfrak)$ denote the center of $\lfrak$.  Then
as vector spaces
$\g=\radu\oplus \z(\lfrak)\oplus [\lfrak,\lfrak]$
where 
the ideal $\radu$ consists of nilpotent elements,
$\z(\lfrak)$ is abelian and consists of semisimple elements,
$[\lfrak,\lfrak]$ is semisimple,
the ideal $[\g,\g]=\radu\oplus [\lfrak,\lfrak]$,
the ideal $\radu\oplus \z(\lfrak)$ equals the radical of $\g$,
and
$\lfrak=\z(\lfrak)\oplus [\lfrak,\lfrak]$.
We thus have:
$$
\begin{matrix}
[\radu,\radu]\subseteq \radu
 & [\radu,[\lfrak,\lfrak]]\subseteq \radu
 & [\radu,\z(\lfrak)]\subseteq \radu \\
 & [[\lfrak,\lfrak],[\lfrak,\lfrak]]=[\lfrak,\lfrak]
 & [[\lfrak,\lfrak],\z(\lfrak)]=0 \\
 & & [\z(\lfrak),\z(\lfrak)]=0.
\end{matrix}$$
\end{proposition}
\begin{proof}
Most of this follows from the algebraic Levi decomposition of $G$.
Since 
$[L,L]$ is semisimple and $\Rad_u(G)$ is unipotent,
their Lie algebras are semisimple and nilpotent, respectively.
Since $\radu$ is nilpotent, all of its elements are nilpotent.
Then apply Corollary \ref{cor:ofhomomorphism2}\eqref{enit:genbracket}.
\end{proof}

\begin{remark}
Theorem 6.1 of \cite{gmns} describes a
normal form for a basis of the logarithmic vector
fields of a linear free divisor.
A key ingredient is a maximal subspace of simultaneously
diagonalizable linear logarithmic vector fields,
i.e., the vector fields corresponding to the
Lie algebra $\tfrak$ of a maximal torus. 
By Remark \ref{rem:boreltori}, $\tfrak$ may always be chosen to
contain $\z(\lfrak)$.
\end{remark}

\subsection{Some special cases}
\label{subsec:specialcases}

We now apply our results to several special types of linear free divisors.

\subsubsection{Abelian groups}
The \emph{normal crossings divisor} in a vector space $V$ is
given by the union of all coordinate hyperplanes for some choice of
vector space coordinates.
It is a linear free divisor,
and is the only linear free divisor defined by an abelian group.

\begin{corollary}[{\cite[Theorem 2.12]{freedivisorsinpvs}}]
Let $V$ be a finite-dimensional complex vector space and suppose that
a connected complex linear algebraic group $G\subseteq \GL(V)$ defines
a linear free divisor $D$ in $V$.
Then $G$ is abelian if and only if $D$ is
equivalent, under a change of coordinates in $V$, to the normal crossings divisor.
\end{corollary}
\begin{proof}
Let $n=\dim(G)=\dim(V)$.

If $D$ is the normal crossings divisor, then after choosing a basis of
$V$, $G$ is the diagonal group in $\GL(V)$, isomorphic to
$(\Gm)^{n}$.  Thus $G$ is abelian.

If $G$ is abelian,
then by Theorem \ref{thm:ofhomomorphism}\eqref{enit:unique1}, $D$ has
$\dim(G)=n$
irreducible components.
By Theorem \ref{thm:ofhomomorphism}\eqref{enit:unique5},
$G$ is isomorphic to $(\Gm)^n$, and hence is a maximal torus in
$\GL(V)$.
As all such tori are conjugate in $\GL(V)$,
we may choose coordinates on $V$ so that $G$ is the group of diagonal
matrices.
A calculation then shows that $D$ is the normal crossings divisor.
\end{proof}

\subsubsection{Irreducible linear free divisors}

We now examine the groups that
produce irreducible
linear free divisors.
%
Observe that if $G\subseteq \GL(V)$ defines a linear free divisor
$D\subset V$, then $\lambda\cdot I\in (\GL(V)_D)^0=G$ 
for $I$ the identity element and
$\lambda\in\C^*$.

Recall that an algebraic group $H$ is called \emph{perfect} if
$[H,H]=H$.  For instance, semisimple groups are perfect.

\begin{corollary}
\label{cor:irreducible}
Let $D\subset V$ be a linear free divisor defined by the group
$G\subseteq \GL(V)$.
Let $H=G\cap \SL(V)$, and let $K=(\C^*)\cdot I\subseteq G$.
The following are equivalent:
\begin{enumerate}
\item
\label{en:1component}
$D$ has $1$ irreducible component.
\item
\label{en:h0gg}
$H^0=[G,G]$.
\item
\label{en:gkgg}
$G=K\cdot [G,G]$.
\item
\label{en:h0perfect}
$H^0$ is perfect.
\item
\label{en:isperfect}
There exists a perfect connected codimension $1$ algebraic subgroup
$J$ of $G$.
\end{enumerate}
When these hold, $J=[G,G]=H^0$.
\end{corollary}
\begin{proof}
Clearly $G=K\cdot H$, and there are a finite number of ways to write
$g\in G$ as a product of elements of $K$ and $H$.
It follows that $\dim(H)=n-1$.
The multiplication morphism $K\times H^0\to G$ has a connected image
of dimension $\dim(G)$, and hence 
$G=K\cdot H^0$.

If
\eqref{en:1component}, then
by Theorem \ref{thm:ofhomomorphism}\eqref{enit:unique1}, 
$\dim([G,G])=n-1$.
Since $[G,G]\subseteq H$ are of the same dimension
and $[G,G]$ is connected, we have $[G,G]=H^0$ and \eqref{en:h0gg}.

If \eqref{en:h0gg}, then by the above work, $G=K\cdot [G,G]$, giving
\eqref{en:gkgg}.

Suppose \eqref{en:gkgg}.
Since $[G,G]\subseteq H$ we have
$\dim([G,G])\leq n-1$.
Also, 
$\dim(G)\leq \dim(K)+\dim([G,G])$
shows $\dim([G,G])\geq n-1$, and hence
$\dim([G,G])=n-1$. 
Since they are connected and of the same dimension, $[G,G]=H^0$.
Since $K$ is in the center of $G$,
we have $[K\cdot N,K\cdot M]=[N,M]$ for any subgroups $N$ and $M$ of $G$.
In particular,
\begin{equation*}
H^0=[G,G]=[K\cdot [G,G],K\cdot [G,G]]=[[G,G],[G,G]]=[H^0,H^0],
\end{equation*}
giving \eqref{en:h0perfect}.

If \eqref{en:h0perfect}, then since $\dim(H^0)=n-1$,
we have \eqref{en:isperfect}.

If \eqref{en:isperfect}, then
since
$J=[J,J]\subseteq [G,G]\subseteq H$ and
$\dim(H)=n-1$,
we have $\dim([G,G])=n-1$.
By Theorem \ref{thm:ofhomomorphism}\eqref{enit:unique1},
$D$ has $1$ irreducible component, proving \eqref{en:1component}.
Finally, note that $J\subseteq [G,G]\subseteq H^0$ are all
connected algebraic groups of the same dimension, hence equal.
\end{proof}

\begin{remark}
The case when $H$ is semisimple was thoroughly explored
in \cite{freedivisorsinpvs};
by the Levi decomposition of $G$, this is equivalent to
$\dim(\gpcenter(L))=1$ and $\Rad_u(G)=\{e\}$.
Are there other irreducible linear free divisors, with $H^0$ perfect and
$\Rad_u(G)\neq \{e\}$?
Since an irreducible representation that is a
prehomogeneous vector space has 
$H$ semisimple
by \cite[Theorem 7.21]{kimura}, 
in such an example the representation must be reducible.
\end{remark}

\subsubsection{Reductive groups}

For reductive groups, we have the following.
\begin{corollary}[{\cite[Lemma 2.6]{freedivisorsinpvs}}]
\label{cor:reductive}
For a linear free divisor $D$ defined by a reductive group $G$, the number of
irreducible components of $D$ equals the dimension of the center of
$G$.
Let $H$ be the subgroup of $G$ leaving invariant the level sets of
the product
$f_1\cdots f_r$.
Then
$D$ is irreducible if and only if $H^0$ is semisimple.
\end{corollary}
\begin{proof}
Let $G$ have a Levi decomposition.
Since $\Rad_u(G)$ is trivial and $G$ is itself a Levi subgroup,
apply Corollary \ref{cor:ofhomomorphism2}\eqref{enit:gencomponents}
to get the first statement.

If $D$ is irreducible, then $r=1$ and
so by Theorem \ref{thm:ofhomomorphism}\eqref{enit:unique3},
$H=[G,G]\cdot G_{v_0}$.  In particular, $H^0=[G,G]$,
and this is semisimple by the structure theory.

Conversely, suppose $H^0$ is semisimple.
Since $H=\ker(\chi_1\cdots \chi_r)$,
by Theorem \ref{thm:homomorphism}\eqref{enit:1},
$H^0$ has codimension $1$ in $G$.
Since $H^0$ is perfect, by Corollary \ref{cor:irreducible},
$D$ is irreducible.
\end{proof}

Granger--Mond--Schulze also show (\cite[Theorem
2.7]{freedivisorsinpvs})
that for a linear free divisor $D\subset V$ defined by a reductive
group $G$,
the number of irreducible hypersurface components of $D$ equals
the number of irreducible $G$--modules in $V$. 

\begin{example}
Let $D\subset V$ be a linear free divisor constructed from a
\emph{quiver} $Q$,
a finite connected oriented graph with vertex set $Q_0$,
edge set $Q_1$, and a dimension vector $d:Q_0\to \N$
(see \cite{gmns,mondbuchweitz,freedivisorsinpvs}).
Here, the group $G$ is a product over $Q_0$ of general linear
groups,
so $G$ is reductive with $\dim(\gpcenter(G))=|Q_0|$,
and $V$ is the space of representations of $(Q,d)$.
When $d$ is a real Schur root and $Q$ has no oriented cycles, then
$G$ has an open orbit and
a theorem of Kac states that
the complement has
$|Q_0|-1$ irreducible hypersurface components
(\cite[\S4]{gmns}).
The apparent disagreement with 
Corollary \ref{cor:reductive} is resolved by  
observing that $G$ does not define $D$ in the sense of
Definition \ref{def:defineslfd} as
the representation $\rho$ of $G$ has a
$1$-dimensional kernel contained in $\gpcenter(G)$,
whereas $\rho(G)$
defines $D$ and
is
reductive with center of dimension $|Q_0|-1$.
\end{example}

\subsubsection{Solvable groups}

%
Recall that by 
the Lie--Kolchin Theorem,
any solvable linear algebraic group $G\subset \GL(V)$
has a basis of $V$
in which 
$G$ is lower triangular.

\begin{corollary}
\label{cor:solvable}
Let $D\subset V$ be a linear free divisor defined by a solvable group
$G\subseteq \GL(V)$.
Fix any basis which makes $G$ lower triangular, and
let $\phi:G\to (\Gm)^{\dim(V)}$ send $g$ to the diagonal entries of
$g$. 
Then:
\begin{enumerate}
\item
\label{enit:solvablenum}
The number of components of $D$ equals the dimension of the maximal torus
of $G$, and also $\dim(\phi(G))$.
\item
\label{enit:solvablehyp}
$D$ has a hyperplane component.
\item
\label{enit:solvableunip}
Let $G_u$ be the subgroup consisting of unipotent elements of $G$.
Then $[G,G]=G_u=\ker(\phi)$.
\item 
\label{enit:solvablefactor}
Every $\chi\in X_1(G)$ factors through $\phi$.
\end{enumerate}
\end{corollary}
\begin{proof}
At first we proceed without the hypothesis that $G$ defines a linear
free divisor.
Since $G$ is connected and solvable, the Levi decomposition of $G$ has
$\Rad(G)=G$, $\Rad_u(G)=G_u$, $L$ is a
maximal torus of $G$, and
in this case $[G,G]\subseteq G_u$. 
Note that $\phi$ is a homomorphism of linear algebraic groups, and by
the definition of unipotent, $G_u=\ker(\phi)$.
Then since $\dim(G)=\dim(G_u)+\dim(L)=\dim(\ker(\phi))+\dim(L)$,
we have
\begin{equation}
\label{eqn:imphistar}
\dim(L)=\dim(\phi(G)).
\end{equation}

Now let $\chi\in X(G)$.  By the Jordan decomposition, $G_u\subseteq
\ker(\chi)$, and hence $\ker(\phi)\subseteq \ker(\chi)$.
If $\overline{\chi}$ and $\overline{\phi}$ are the induced
homomorphisms on $G/G_u$, then the homomorphism
$\lambda=\overline{\chi}\circ (\overline{\phi})^{-1}:\phi(G)\to \Gm$
satisfies $\chi=\lambda\circ \phi$.
Since
$\lambda$
is a character on a subtorus of $(\Gm)^{\dim(V)}$,
by \cite[8.2]{borel}
$\lambda$ extends to a character $\psi:(\Gm)^{\dim(V)}\to \Gm$
with $\chi=\psi\circ \phi$.
This proves \eqref{enit:solvablefactor}.

Now assume that $G$ defines a linear free divisor $D$.
By Corollary \ref{cor:ofhomomorphism2}\eqref{enit:gencomponents},
the number of components of $D$ equals $\dim(\gpcenter(L))$.
Then the observation that $L=\gpcenter(L)$ and \eqref{eqn:imphistar} implies
\eqref{enit:solvablenum}.
By Theorem \ref{thm:ofhomomorphism}\eqref{enit:unique4} we have
$G_u\subseteq [G,G]$, and hence $[G,G]=G_u=\ker(\phi)$, proving
\eqref{enit:solvableunip}.

Finally, 
the Lie--Kolchin Theorem guarantees an invariant complete flag in $V$,
hence an invariant hyperplane $H$.
As $H$ cannot
intersect $\Omega$, $H$ is a component of $D$, proving
\eqref{enit:solvablehyp}.
In these lower-triangular coordinates, this $H$ is defined by the
``first coordinate''.
\end{proof}

\begin{remark}
By Corollary \ref{cor:solvable}\eqref{enit:solvablefactor}, the characters
corresponding to the basic relative invariants
are functions of the
diagonal entries.
\end{remark}

\subsection{Degrees of the components}
\label{subsec:whatelse}
Let $G$ define a linear free divisor $D$
with basic relative invariants $f_1,\ldots,f_r$.
We have seen that $r$
may be computed from
the Lie algebra $\g$ of $G$;
may the degrees of the $f_i$ be computed from $\g$?

If we only use the abstract Lie algebra structure of $\g$, then the
answer is no:
\begin{example}
Consider the following two linear free divisors in $\C^5$:
\begin{align*}
D_1:\quad & (x_3 x_5-x_4^2)\begin{vmatrix} 0 & x_1 & x_2 \\ x_1 & x_3 & x_4
\\ x_2 & x_4 & x_5 \end{vmatrix}=0, \\
\text{and }D_2:\quad &
(x_2^2x_3^2-4x_1x_3^3-4x_2^3x_4+18x_1x_2x_3x_4-27x_4^2x_1^2)x_5=0.
\end{align*}
This $D_1$ is \cite[Example 9.4]{DP-matrixsingI},
while $D_2$ is the product-union of a hyperplane with
\cite[Theorem 2.11(2)]{freedivisorsinpvs}.
The degrees of the polynomials defining the
irreducible components of $D_1$ and $D_2$
differ.
However,
the groups defining $D_1$ and $D_2$ have the same abstract Lie algebra
structure:
$\gl_2(\C)\oplus \gl_1(\C)$.
Thus, $D_1$ and $D_2$ are constructed from inequivalent
representations of the same abstract Lie algebra.
\end{example}

Of course, since 
the representation of the Lie algebra determines the divisor,
the representation undoubtedly
contains the information necessary to compute these degrees.
How may we do so effectively?

\begin{lemma}
\label{lemma:degrees}
Let $G\subseteq \GL(V)$ define the linear free divisor $D$ with
basic relative invariants $f_1,\ldots,f_r$, with $f_i\associatedmult
\chi_i$.
Choose $X_i\in\g\subseteq \gl(V)$ such that $d(\chi_i)_{(e)}(X_j)=\delta_{ij}$, or
equivalently, $\xi_{X_i}(f_j)=\delta_{ij} f_j$.
Let $I$ denote the identity endomorphism in
both $G\subseteq \GL(V)$ and
$\g\subseteq\gl(V)$, with $\xi_I$ the Euler
vector field.
Let the module $A\subseteq \Der_V$ consist of the vector fields
annihilating all $f_i$.
Then:
\begin{align}
\xi_I(f_i)&=\deg(f_i) f_i,
  \label{eqn:degree3} \\
d(\chi_i)_{(e)}(I)&=\deg(f_i),
  \label{eqn:degree0} \\
\chi_i(\lambda\cdot I)&=\lambda^{\deg(f_i)}\quad
  \text{for all $\lambda\in\C^*$},
  \label{eqn:degree4} \\
I&=\sum_{j=1}^r \deg(f_j) X_j\bmod{[\g,\g]},
  \label{eqn:degree1} \\
\text{and }
\xi_I&=\sum_{j=1}^r \deg(f_j) \xi_{X_j}\bmod{A}.
  \label{eqn:degree2}
\end{align}
\end{lemma}
\begin{proof}
Corollary \ref{cor:homomorphismliealgebra} shows that such $X_i$
exist.
By \eqref{eqn:diffrelinv}, the $\xi_{X_i}$ have the claimed effects on $f_j$ for
all $i,j$.
Then \eqref{eqn:degree3} is just the Euler relation,
and applying \eqref{eqn:diffrelinv} shows 
\eqref{eqn:degree0}.
Integrating \eqref{eqn:degree0} shows \eqref{eqn:degree4}.
By \eqref{eqn:degree0}, for each $i$ we have
$$d(\chi_i)_{(e)}\Big(I-\sum_{j=1}^r \deg(f_j)
X_j\Big)=\deg(f_i)-\deg(f_i)=0,$$
and since $\cap_{i=1}^r \ker(d(\chi_i)_{(e)})=[\g,\g]$ by
Theorem \ref{thm:homomorphism}\eqref{enit:1},
we have \eqref{eqn:degree1}.
A similar argument using \eqref{eqn:degree3}
shows that $\xi_I-\sum_{j=1}^r \deg(f_j) \xi_{X_j}$ annihilates each $f_i$,
giving \eqref{eqn:degree2}.
\end{proof}

It is unclear whether the embedding $\g\subseteq \gl(V)$ may be used
to find the degrees without first finding either the $\chi_i$,
$d(\chi_i)_{(e)}$, or $f_i$.

\begin{remark}
For $d\in\N$ let $R_d\subseteq \Gm$ denote the group of $d$th roots of
unity.  Then in the situation of Lemma \ref{lemma:degrees}
with $G\subseteq \GL(V)$,
\eqref{eqn:degree4} implies that
$$
\ker(\chi_i|_{\Gm\cdot I})=\ker(\chi_i)\cap (\Gm\cdot I)
 = R_{\deg(f_i)}\cdot I.$$
Then by Theorem \ref{thm:homomorphism}\eqref{enit:1} and number
theory,
$([G,G]\cdot G_{v_0})\cap (\Gm\cdot I)=R_{d}\cdot I$ for
$d$ the greatest common divisor of 
$\{\deg(f_1),\ldots,\deg(f_r)\}$.
\end{remark}

\subsection{Homotopy groups of \texorpdfstring{$V\setminus D$}{V-D}}
\label{subsec:homotopy}
The results above
can
give some insight into the topology of the complement of a linear free
divisor $D$ 
and two types of (global) Milnor fiber associated to $D$.

\begin{proposition}
\label{prop:homotopy}
Let $G\subseteq \GL(V)$ define a linear free divisor $D\subset V$ with
irreducible components defined by $f_1,\ldots,f_r$,
and $f_i\associatedmult \chi_i$.
Let $L$ be a Levi factor of $G$.
Let $v_0\in\Omega$, let $F=(f_1,\ldots,f_r):V\to\C^r$, let
$K$ be the fiber of $f_1\cdots f_r$ containing $v_0$,
and let $P\subset K$ be the fiber of $F$ containing $v_0$. 
Use $v_0$ or $e\in G$ as the base point for all homotopy groups.
Then
for $n>1$,
$$
\pi_n(L)\cong
\pi_n([L,L])\cong
\pi_n(G)\cong
\pi_n([G,G])\cong
\pi_n(P)\cong
\pi_n(\Omega)\cong
\pi_n(K),$$
and for $n=1$ we have the exact
diagrams in 
Figure \ref{fig:pi1}.
\end{proposition}
\begin{figure}[ht]
\caption{The exact diagrams describing the fundamental groups
of the spaces in Proposition \ref{prop:homotopy}.
For a map $\phi$, let $\phi_*$ denote the associated map of
fundamental groups.  Each $i$ is constructed from an inclusion map.
We do not claim commutativity of these diagrams.} 
\label{fig:pi1}
\begin{gather*}
\xymatrix@-1pc{
& & & 0 \ar[d] & 0 \ar[dl] \\
& & & \pi_1(P) \ar[dl]_{i_*} \ar[d]^{i_*} \\
& 0 \ar[r] & \pi_1(K)\ar[dl]^{{p_8}_*} \ar[r]^{i_*} & \pi_1(\Omega)
\ar[d]^{{p_7}_*} \ar[rr]^-{(m\circ p_7)_*} & & \Z
\ar@{=}[d] \ar[r] & 0 \\
0\ar[r] & \Z^{r-1} \ar[dl] \ar[rr] & & \Z^r \ar[rr]^{m_*} \ar[d] & & \Z \ar[r] & 0 \\
0 & & & 0 & & 
}
\\
\xymatrix@-1pc{
& 0 \ar[d] & 0 \ar[d] \\
0 \ar[r] & \pi_1([G,G]) \ar[r]^-{{p_2}_*} \ar[d]_{i_*} & \pi_1(P)
\ar[r]
\ar[d]^{i_*} &
[G,G]\cap G_{v_0} \ar[r] & 0 \\
0 \ar[r] & \pi_1(G) \ar[r]^{{p_1}_*} \ar[d]_{{p_5}_*} & \pi_1(\Omega) \ar[r] \ar[d]^{{p_7}_*} & G_{v_0} \ar[r] & 0 \\
& \Z^r \ar[d] & \Z^r \ar[d] \\ 
& 0 & 0
}
\\
\xymatrix@-1pc{
& 0 \ar[d]& 0\ar[d] & \\
0 \ar[r] & \pi_1([G,G]) \ar[r]^-{i_*} \ar[d]_{{p_4}_*} & \pi_1(G)
\ar[r]^-{{p_5}_*} \ar[d]^{{p_3}_*} & \Z^r
\ar@{=}[d] \ar[r] & 0 \\
0 \ar[r] & \pi_1([L,L]) \ar[r]_-{i_*} \ar[d] & \pi_1(L) \ar[r]_-{{p_6}_*} \ar[d]& \Z^r \ar[r] & 0 \\
& 0 & 0
}
\end{gather*}
\end{figure}
\begin{proof}
If $G_1\subset G_2\subset G_3$ are Lie groups, then
by \cite[\S7.4--7.5]{steenrod}
the map of cosets 
$G_3/G_1\to G_3/G_2$ 
is a fiber bundle with fiber $G_2/G_1$.
Thus, we have the following fiber bundles:
\begin{align*}
p_1:G    &\to G/G_{v_0}\cong \Omega
&
p_2:[G,G]&\to [G,G]/([G,G]\cap G_{v_0})
\\
p_3:G    &\to G/\Rad_u(G)\cong L
&
p_4:[G,G]&\to [G,G]/\Rad_u(G)\cong [L,L]
\\
p_5:G&\to G/[G,G]\cong \Gm^r
&
p_6:L&\to L/[L,L]\cong \Gm^r \\
q:G/G_{v_0}&\to G/[G,G]\cdot G_{v_0}
\end{align*}
The identifications of the codomains
of $p_1$, $p_3$, $p_4$, $p_5$, and $p_6$
are by, respectively,
the orbit map $\alpha_{v_0}$ of $v_0$,
the definition of the Levi decomposition,
Corollary \ref{cor:ofhomomorphism2}\eqref{enit:genbracket},
Theorem \ref{thm:ofhomomorphism}\eqref{enit:unique5},
and
the isomorphism
$$L/[L,L]\cong (G/\Rad_u(G))/([G,G]/\Rad_u(G))\cong G/[G,G]$$
that follows from the Levi decomposition and
Corollary \ref{cor:ofhomomorphism2}\eqref{enit:genbracket}.

Since $G_{v_0}$ and $[G,G]\cap G_{v_0}$ are finite,
$p_1$ and $p_2$ are covering spaces with deck transformation groups
isomorphic to $G_{v_0}$ and $[G,G]\cap G_{v_0}$, respectively
(\cite[Proposition 1.40]{hatcher}).

For $q$, we use the orbit map to identify
$G/G_{v_0}$ with $\Omega$ and
then by Theorem \ref{thm:homomorphism}\eqref{enit:1}
compose with the isomorphism $G/[G,G]\cdot G_{v_0}\to (\Gm)^r$
induced by 
$(\chi_1,\ldots,\chi_r):G\to (\Gm)^r$.
This gives a fiber bundle
$$p_7:\Omega\to (\Gm)^r$$
defined by
$p_7(\rho(g)(v_0))=(\chi_1,\ldots,\chi_r)(g)$
with fiber homeomorphic to $\ker(\chi_1,\ldots,\chi_r)/G_{v_0}=[G,G]\cdot
G_{v_0}/G_{v_0}$.
As the action of $[G,G]$ on $[G,G]\cdot G_{v_0}/G_{v_0}$
is 
smooth and transitive, and the isotropy subgroup at 
$e G_{v_0}$ is $[G,G]\cap G_{v_0}$,
the fiber of $p_7$ is
isomorphic to $[G,G]/([G,G]\cap G_{v_0})$.

Let $m:(\Gm)^r\to \Gm$ and $n:(\C^*)^r\to \C^*$
both be defined by $(a_1,\ldots,a_r)\mapsto a_1\cdots a_r$.
By \eqref{eqn:relinv} we have a commutative diagram
\begin{equation}
\label{eqn:somediagramofmult}
\xymatrix@-0.5pc{
 &
(\Gm)^r \ar[r]^m \ar[d]^{\beta} & \Gm \ar[d]^{\gamma} \\
\Omega \ar[ur]^{p_7} \ar[r]^F \ar@/_1pc/[rr]_{f_1\cdots f_r}
 & (\C^*)^r \ar[r]^{n} & \C^*}
\end{equation}
where 
$\beta(a_1,\ldots,a_r)=(a_1 f_1(v_0),\ldots,a_r f_r(v_0))$
and $\gamma(a)=a (f_1\cdots f_r)(v_0)$ are homeomorphisms.
By \eqref{eqn:somediagramofmult}, $P$ is
homeomorphic to the fiber of $p_7$, that is,
$P\cong [G,G]/([G,G]\cap G_{v_0})$,
and hence $P$ is the codomain of $p_2$. 
Also, $K$ is homeomorphic
to
$(m\circ p_7)^{-1}(1)$;
restricting $p_7$ gives a fiber bundle
$$p_8:K\to \ker(m)\cong (\Gm)^{r-1}$$
with fiber $P$.

By \cite[Proposition 4.48]{hatcher},
a fiber bundle is a Serre fibration, that is,
it possesses the homotopy lifting property for CW complexes,
and these have the usual homotopy long exact sequence of a fibration
(\cite[Theorem 4.41]{hatcher});
apply this sequence 
to all $p_i$ and $m\circ p_7$.
Note that the connected unipotent group $\Rad_u(G)$
is diffeomorphic to some $\C^p$, and hence is contractible
(\cite[\S3.3.6]{ov}).
Also,
$G_{v_0}$ and $[G,G]\cap G_{v_0}$ are finite,
$P\cong [G,G]/[G,G]\cap G_{v_0}$ is connected,
$\pi_1((\Gm)^k)\cong\Z^k$, and $\pi_i((\Gm)^k)\cong 0$ for $i> 1$.
%
%

The long exact sequences show that for $n>1$
there is a diagram, not necessarily commutative,
%
$$
\xymatrix@-1pc{
\pi_n([L,L])\ar[d]_{p_6}
&
\ar[l]_{p_4} \pi_n([G,G]) \ar[r]^-{p_2} \ar[d]_{p_5}
&
\pi_n(P) \ar[r]^{p_8} \ar[d]_{p_7}
&
\pi_n(K) \ar[dl]^{m\circ p_7} 
\\
\pi_n(L)
&
\ar[l]_{p_3} \pi_n(G) \ar[r]^{p_1}
&
\pi_n(\Omega)
&
}$$
where an arrow labeled by $\phi$ represents
an isomorphism
occurring 
in the long exact sequence of the fibration $\phi$,
either $\phi_*$ or $i_*$ for $i$ the inclusion of the
fiber.
The sequences also give
the exact
diagrams in Figure~\ref{fig:pi1}.
\end{proof}

Thus, the homotopy groups may largely be computed from the homotopy
groups of the semisimple part $[L,L]$ of $G$.
For instance, if $G$ is solvable then $[L,L]=\{e\}$ and hence
by
Proposition \ref{prop:homotopy},
$\Omega$, $P$, and $K$ are $K(\pi,1)$ spaces, as shown in \cite{kpi1}.

\subsection{(Non-linear) free divisors}

Let $(D,p)$ be an arbitrary free divisor in $(\C^n,p)$.
May the number of components of
$(D,p)$ be computed from the structure of $M=\Derlog{\C^n,p}{D}$?
One natural guess,
$\dim_\C(M/\mathscr{O}_{\C^n,p}\cdot [M,M])$,
does not work.
For example, if $D$ is a hyperplane in $\C^2$, then the number
computed is $0$; in fact, $M=\mathscr{O}_{\C^n,p}\cdot [M,M]$ whenever $\Derlog{\C^n,p}{D}\nsubseteq
\mathscr{M}_p\cdot \Der_{\C^n,p}$ for $\mathscr{M}_p$ the maximal
ideal.
Other examples give answers too large:
for the plane curve $D=V((a^2-b^3)(a^7-b^{13}))$,
the number computed is $35$.

\bibliographystyle{amsalpha}
\bibliography{refs}

\providecommand{\bysame}{\leavevmode\hbox to3em{\hrulefill}\thinspace}
\providecommand{\MR}{\relax\ifhmode\unskip\space\fi MR }
\providecommand{\MRhref}[2]{%
  \href{http://www.ams.org/mathscinet-getitem?mr=#1}{#2}
}
\providecommand{\href}[2]{#2}
\begin{thebibliography}{GMNRS09}

\bibitem[BM06]{mondbuchweitz}
Ragnar-Olaf Buchweitz and David Mond, \emph{Linear free divisors and quiver
  representations}, Singularities and computer algebra, London Math. Soc.
  Lecture Note Ser., vol. 324, Cambridge Univ. Press, Cambridge, 2006,
  pp.~41--77. \MR{2228227 (2007d:16028)}

\bibitem[Bor91]{borel}
Armand Borel, \emph{Linear algebraic groups}, second ed., Graduate Texts in
  Mathematics, vol. 126, Springer-Verlag, New York, 1991. \MR{1102012
  (92d:20001)}

\bibitem[Bri06]{brion}
Michel Brion, \emph{Some remarks on linear free divisors}, E-mail to
  Ragnar-Olaf Buchweitz, September 2006.

\bibitem[DP12a]{kpi1}
James Damon and Brian Pike, \emph{Solvable group representations and free
  divisors whose complements are {$K(\pi,1)$}'s}, Topology Appl. \textbf{159}
  (2012), no.~2, 437--449. \MR{2868903}

\bibitem[DP12b]{DP-matrixsingI}
James Damon and Brian Pike, \emph{Solvable groups, free divisors and
  nonisolated matrix singularities {I}: Towers of free divisors}, Submitted.
  \href{http://arxiv.org/abs/1201.1577}{arXiv:1201.1577 [math.AG]}, 2012.

\bibitem[GMNRS09]{gmns}
Michel Granger, David Mond, Alicia Nieto-Reyes, and Mathias Schulze,
  \emph{Linear free divisors and the global logarithmic comparison theorem},
  Ann. Inst. Fourier (Grenoble) \textbf{59} (2009), no.~2, 811--850.
  \MR{2521436 (2010g:32047)}

\bibitem[GMS11]{freedivisorsinpvs}
Michel Granger, David Mond, and Mathias Schulze, \emph{Free divisors in
  prehomogeneous vector spaces}, Proc. Lond. Math. Soc. (3) \textbf{102}
  (2011), no.~5, 923--950. \MR{2795728}

\bibitem[Hat02]{hatcher}
Allen Hatcher, \emph{Algebraic topology}, Cambridge University Press,
  Cambridge, 2002. \MR{1867354 (2002k:55001)}

\bibitem[Hum90]{humphreys-reflection_groups}
James~E. Humphreys, \emph{Reflection groups and {C}oxeter groups}, Cambridge
  Studies in Advanced Mathematics, vol.~29, Cambridge University Press, 1990.

\bibitem[Kim03]{kimura}
Tatsuo Kimura, \emph{Introduction to prehomogeneous vector spaces},
  Translations of Mathematical Monographs, vol. 215, American Mathematical
  Society, Providence, RI, 2003, Translated from the 1998 Japanese original by
  Makoto Nagura and Tsuyoshi Niitani and revised by the author. \MR{1944442
  (2003k:11180)}

\bibitem[OV90]{ov}
A.~L. Onishchik and {\`E}.~B. Vinberg, \emph{Lie groups and algebraic groups},
  Springer Series in Soviet Mathematics, Springer-Verlag, Berlin, 1990,
  Translated from the Russian and with a preface by D. A. Leites. \MR{1064110
  (91g:22001)}

\bibitem[Pik]{pike-generalizebrion}
Brian Pike, \emph{On {F}itting ideals of logarithmic vector fields and
  {S}aito's criterion}, \href{http://arxiv.org/abs/1309.3769}{arXiv:1309.3769
  [math.AG]}.

\bibitem[Sat90]{sato}
Mikio Sato, \emph{Theory of prehomogeneous vector spaces (algebraic part)---the
  {E}nglish translation of {S}ato's lecture from {S}hintani's note}, Nagoya
  Math. J. \textbf{120} (1990), 1--34, Notes by Takuro Shintani, Translated
  from the Japanese by Masakazu Muro. \MR{1086566 (92c:32039)}

\bibitem[Spr98]{springer}
T.~A. Springer, \emph{Linear algebraic groups}, second ed., Progress in
  Mathematics, vol.~9, Birkh\"auser Boston Inc., Boston, MA, 1998. \MR{1642713
  (99h:20075)}

\bibitem[Ste51]{steenrod}
Norman Steenrod, \emph{The {T}opology of {F}ibre {B}undles}, Princeton
  Mathematical Series, vol. 14, Princeton University Press, Princeton, N. J.,
  1951. \MR{0039258 (12,522b)}

\end{thebibliography}

\end{document}